\documentclass{amsart}
\usepackage{amssymb}
\usepackage{textcomp}
\usepackage[mathscr]{euscript}
\usepackage{enumerate}
\usepackage[dvipsnames]{xcolor}
\usepackage{pb-diagram}
\usepackage{bbm}

\usepackage[ngerman, english]{babel}
\usepackage[latin1]{inputenc}
\usepackage{tikz}
\usetikzlibrary{shapes,arrows}
\usepackage{verbatim}

\newtheorem{theorem}{Theorem}[section]

\newtheorem{thm}[theorem]{Theorem}
\newtheorem{prop}[theorem]{Proposition}
\newtheorem{lem}[theorem]{Lemma}

\newtheorem{fact}[theorem]{Fact}
\newtheorem{cor}[theorem]{Corollary}

\newtheorem{corollary}[theorem]{Corollary}
\newtheorem*{conj}{Conjecture}
\newtheorem*{thmA}{Theorem A}

\newtheorem*{thmB}{Theorem B}

\theoremstyle{definition}
\newtheorem{defn}[theorem]{Definition}

\newtheorem{example}[theorem]{Example}

\theoremstyle{remark}

\ProvideTextCommandDefault{\cprime}{(U+042C)}

\newcommand{\ind}{\hspace{3ex}}

\newcommand{\ldim}{\ensuremath{\textup{ldim}}}
\newcommand{\td}{\ensuremath{\textup{td}}}

\newcommand{\cal}[1]{\ensuremath{\mathcal{#1}}}

\newcommand{\FO}{\operatorname{FO}}

\newcommand{\Z}{\mathbb{Z}}
\newcommand{\N}{\mathbb{N}}
\newcommand{\C}{\mathbb{C}}
\newcommand{\Q}{\mathbb{Q}}
\newcommand{\R}{\mathbb{R}}

\newcommand{\Member}{{\bf Member} }
\newcommand{\Admissible}{{\bf Admissible} }
\newcommand{\lcm}{\ensuremath{\textup{lcm}}}
\newcommand{\sinPA}{\boldsymbol{\sin}\textbf{-\textup{PA}}}

\begin{document}
\title{Decidability bounds for Presburger arithmetic extended by sine}



\author{Eion Blanchard }
\address
{Department of Mathematics, University of Illinois at Urbana-Champaign, 1409 West Green Street, Urbana, IL 61801}
\email{eionmb2@illinois.edu}

\author{Philipp Hieronymi}
\address{Mathematisches Institut, Universit\"{a}t Bonn, Endenicher Allee 60, D-53115 Bonn, Germany}
\email{hieronymi@math.uni-bonn.de}


\maketitle

\begin{abstract}
We consider Presburger arithmetic extended by the sine function, call this extension sine-Presburger arithmetic ($\sinPA$), and systematically study decision problems for sets of sentences in $\sinPA$.
In particular, we detail a decision algorithm for existential $\sin$-PA sentences under assumption of Schanuel's conjecture. This procedure reduces decisions to the theory of the ordered additive group of real numbers extended by sine, which is decidable under Schanuel's conjecture.
On the other hand, we prove that four alternating quantifier blocks suffice for undecidability of $\sin$-PA sentences. To do so, we explicitly interpret the weak monadic second-order theory of the grid, which is undecidable, in $\sinPA$.
\end{abstract}

\section{Introduction}

 A \textbf{sine}-\textbf{Presburger sentence} ($\sin$-PA sentence) is a statement of the form \begin{equation*} Q_1 x_1 \in \Z^{n_1} \dots Q_r x_r \in \Z^{n_r} \ \Phi\left(x_1, \dots, x_r\right), \end{equation*} where $Q_1, \dots, Q_r \in \left\{ \forall, \exists \right\}$ are $r$ alternating quantifiers, and $\Phi$ is a Boolean combination of linear-sine inequalities in $x_1, \dots, x_r$, which are inequalities of terms built up from variables and rational numbers by nested applications of sine, addition, and rational scalar multiplication.
 Further consider the signature \[ \mathcal{L}_{\sin} = \{<,+,0,1,\sin, (\lambda_q)_{q\in\Q} \} \] where $\lambda_q : x \mapsto qx$.
 Then equivalently, a $\sin$-PA sentence is a first-order $\mathcal{L}_{\sin}$-sentence in prenex normal form with quantification restricted to $\Z$.
 \newline 
 
 \noindent The goal of this paper is to systematically study decision problems for sets of $\sin$-PA sentences.  This is in part motivated by a similar analysis of an  extension of Presburger arithmetic by algebraic scalar multiplication in Hieronymi, Nguyen, and Pak \cite{alphaPA}, although our focus here is decidability rather than computational complexity.
 Decision procedures for linear arithmetic with trigonometric functions are applicable to problems in industrial engineering, particularly for linear-trigonometric hybrid and cyber-physical systems as well as for signal processing (see \cite{Platzer}, \cite{Boyd}).
 \newline
 
 \noindent  Of course, by decidability of Presburger arithmetic (see \cite{Presburger}), the truth of a $\sin$-PA sentence in which sine does not appear can be decided. However, this result does not extend to all $\sin$-PA sentences. Similarly to \cite[Theorem 1.5]{alphaPA}, we show that the set of all $\sin$-PA sentences with just four quantifier alternations is undecidable.

\begin{thmA}
The set of $\exists^K \forall^K \exists^K \forall^K\ \sin$-PA sentences, where $K=3388$, is undecidable.
\end{thmA}

\noindent By Hieronymi and Tychonievich \cite[Theorem D]{MR3223381}, the first-order theory \linebreak $\FO\left(\R, <, +, \sin, \Z\right)$ is undecidable. It is not hard to see that this theory contains all true $\sin$-PA sentences. While this alone does not yield undecidability of all $\sin$-PA sentences, the proof in \cite{MR3223381} can be adjusted to do so. This was not explicitly stated in \cite{MR3223381}, so we give a full argument with substantially improved bounds on the number of quantifier alternations that yield undecidability.\newline

\noindent To our knowledge, the largest fragment of $\sinPA$ known to be decidable is the set of existential $\sin$-PA sentences whose appearances of sine all share the same argument (and thus cannot be nested). The explicit decision procedure given by Anai and Weispfenning \cite{MR1805108} considers such sentences as mixed real-integer linear-trigonometric problems with a single variable standing in for the shared sine argument, then employs linear quantifier elimination with symbolic test points.
Using different techniques, we extend this result by establishing the decidability of all existential $\sin$-PA sentences under a far-reaching number-theoretic conjecture.

\begin{thmB}
Assume Schanuel's conjecture holds. Then the set of existential $\sin$-PA sentences is decidable.  
\end{thmB}

\noindent When the sine function is replaced by multiplication in the definition of $\sinPA$, the corresponding statement of Theorem B fails due to the negative solution of Hilbert's 10th problem (e.g., see \cite{MR0432534}). Indeed, even when we replace the sine function by the natural logarithm, the analogue to Theorem B fails as we prove by a reduction to Hilbert's 10th problem in Section \ref{sec:conclusion}. 
\newline

\tikzstyle{decision} = [diamond, draw, fill=blue!20, 
    text width=5.5em, text badly centered, node distance=3cm, inner sep=0pt]
\tikzstyle{block} = [rectangle, draw, fill=blue!20, 
    text width=8em, text centered, rounded corners, minimum height=3.5em]
\tikzstyle{line} = [draw, -latex']
\tikzstyle{cloud} = [draw, ellipse,fill=red!20, node distance=3cm,
    minimum height=2em]
    
\begin{figure}
\begin{center}
\begin{tikzpicture}[node distance=2cm,
    every node/.style={fill=white, font=\sffamily}, align=center]
    
    \node [block, fill=orange!30, text width=12em, minimum height=2.5em] (input) {existential $\sin$-PA sentence};
    \node [block, below of=input, fill=magenta!30, text width=6em] (equalities) {remove equalities};
    \node [block, left of=equalities,  xshift=-1.5cm, draw=blue!80, double, text width=10em] (sine-equalities) {eliminate sine from equalities};
    \node [block, right of=equalities, xshift=1.5cm, fill=magenta!30, text width=10em] (linear-inequalities) {eliminate linear terms from inequalities};
    \node [block, fill=magenta!30, below of=equalities, text width=12em] (M) {construct $\mathcal{L}_{\sin}$-definable real proxy solution set};
    \node [decision, below of=M, draw=blue!80, double] (M-empty) {is proxy set nonempty?};
    \node [cloud, right of=M-empty, xshift=1cm, fill=green!30] (T) {True};
    \node [cloud, left of=M-empty, xshift=-1cm, fill=red!30] (F) {False};
    
    \path [line] (input) -- (sine-equalities);
    \path [line] (sine-equalities) -- (equalities);
    \path [line] (equalities) -- (linear-inequalities);
    \path [line] (linear-inequalities) -- (M);
    \path [line] (M) -- (M-empty);
    \path [line] (M-empty) -- node {no} (F);
    \path [line] (M-empty) -- node {yes} (T);
    
\end{tikzpicture}
\end{center}
\caption{Decision procedure for existential $\sin$-PA sentences. The blue, double-bordered steps assume Schanuel's conjecture.} \label{fig:diag}
\end{figure}
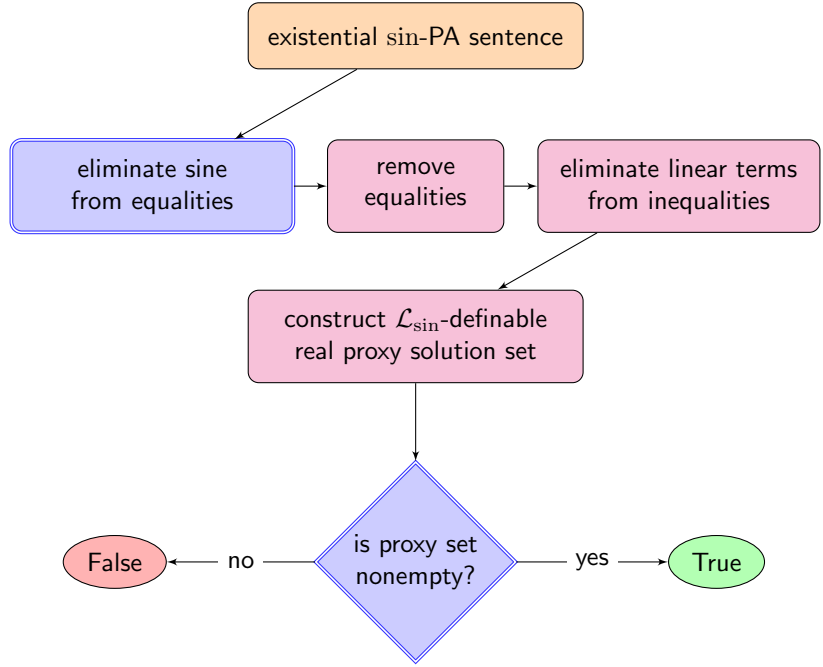

\noindent We now outline the proof of Theorem B, which is portrayed by Figure \ref{fig:diag}.
While $\FO\left(\R, <, +, \cdot, \sin\right)$ is undecidable, the theory $\FO\left(\R, <, +, \cdot, \sin_{[-\pi,\pi]}\right)$ is decidable under Schanuel's conjecture by Macintyre and Wilkie \cite{MR1435773} (see \cite[Theorem 3.1]{MR3522649}). Following the argument by Toffalori and Vozoris \cite[Theorem 2.7]{MR2582162}, we deduce the decidability of $\FO\left(\R, <, +, \sin\right)$ under the same hypothesis\footnote{To our knowledge, even this result has not appeared in the literature, though we assume it has been known to specialists.}.
We reduce deciding existential $\sin$-PA sentences to deciding those in this theory. Recall that a matrix formula is the quantifier-free part of a formula in prenex normal form. Given an existential $\sin$-PA sentence, we construct a quantifier-free $\mathcal{L}_{\sin}$-formula which has a solution over the real numbers if and only if the original matrix formula has a solution over the integers (i.e., the $\sin$-PA sentence at hand is true).
To this end we eliminate appearances of sine from equalities, eradicate equalities by variable replacement, remove linear appearances of variables in the remaining inequalities, then finally arrive to the desired formula by applying the density of representatives for integers modulo $2N\pi$ in the interval $[0,2N\pi)$ for some positive integer $N$.
The correctness of this construction, though not the construction itself, separately invokes Schanuel's conjecture.

\subsection*{Acknowledgements} The authors were partially supported by NSF grant DMS-1654725. 
Both authors thank the Fields Institute for its hospitality during the Thematic Program on Tame Geometry, Transseries, and Applications to Analysis and Geometry. The authors also thank Pantelis Eleftheriou and Chris Miller for discussions around the topic of this paper.

\section{Preliminaries}

\subsection{Notation} Throughout, $i$ and $j$ denote natural numbers. We use $\mathbbm{i}$ for the imaginary unit. Definable means definable without parameters. Lowercase $x$ is reserved for a variable. For a set of natural numbers $S \subseteq \N$, we write $\sin(S)$ as shorthand for the set $\{ \sin s : s \in S \}$.

\subsection{Formalizing sine-Presburger arithmetic}\label{sec:formalize-sinPA}

A $\boldsymbol{\sin}$-\textbf{PA formula} is 
\begin{equation*}
Q_1 x_1 \in \Z^{n_1} \dots Q_r x_r \in \Z^{n_r} \ \Phi\left(x_1, \dots, x_r, y\right),
\end{equation*} 
where $Q_1, \dots, Q_r \in \left\{ \forall, \exists \right\}$ are $r$ alternating quantifiers, and the matrix formula $\Phi$ is a quantifier-free formula in the signature $\mathcal{L}_{\sin}$ with free variables $x_1, \dots, x_r, y$.
A $\sin$-PA formula without free variables (i.e., in which every variable is bounded by a quantifier) is a $\sin$-PA sentence.

Let $\mathcal{L}_{\sin,\Z} = \mathcal{L}_{\sin} \cup \{\Z\}$ and consider the $\mathcal{L}_{\sin,\Z}$-structure $\left( \R, <, +, \sin, \Z\right)$.
Since we will always work over $\left(\R, <, +, \sin, \Z\right)$, we say an $\mathcal{L}_{\sin,\Z}$-sentence is \textbf{true} if it is modeled by $\left(\R, <, +, \sin, \Z\right)$.
For every $\sin$-PA formula $\Phi$, we define the $\mathcal{L}_{\sin,\Z}$-formula $\widetilde{\Phi}$ recursively by
\[ \widetilde{\Phi} := \begin{cases}
    \Phi & \text{$\Phi$ is quantifier-free} \\
    \exists x\ \big( \Z(x) \wedge \widetilde{\Psi} \big) & \text{$\Phi$ is $\exists x\ \Psi$} \\
    \forall x\ \big( \Z(x) \rightarrow \widetilde{\Psi} \big) & \text{$\Phi$ is $\forall x\ \Psi$} .
\end{cases} \]
Let \textbf{sine}-\textbf{Presburger arithmetic} ($\sinPA$) be the set of all true $\sin$-PA sentences. Note that for every $\sin$-PA sentence $\Phi$, \[ \Phi \in \sinPA \text{ if and only if } \left(\R, <, +, \sin, \Z\right) \models \widetilde{\Phi} .\]
In this sense, $\sinPA$ restricts $\FO\left(\R, <, +, \sin, \Z\right)$ to quantification over the integers so that $\sinPA$ is the extension of Presburger arithmetic by the sine function.
From now on, we will identify $\Phi$ and $\widetilde{\Phi}$; it will be clear from context whether $\Phi$ is a $\sin$-PA formula or an $\mathcal{L}_{\sin,\Z}$-formula.

Given a $\sin$-PA formula $\Phi$ with free variables $x$ and $X \in \Z^{|x|}$, we say $\Phi(X)$ is true or that $X$ satisfies $\Phi(x)$ if the $\sin$-PA sentence obtained by replacing $x$ in $\Phi$ by $X$ is true. When we write $\Phi(x) \in \sinPA$, we mean $\forall x \in \Z^n\ \Phi(x) \in \sinPA$.
We say $S \subseteq \Z^m$ is $\boldsymbol{\sin}$-\textbf{PA definable} if there is a $\sin$-PA formula $\Phi(x)$ such that \[ S = \{ X \in \Z^m : \Phi(X) \} . \]


\subsection{Number-theoretic theorems and conjectures}
We first collect number-theoretic results and conjectures needed in this paper.

\begin{theorem}[Lindemann-Weierstrass, Baker's formulation \cite{MR1074572}]\label{thm:baker}
If $\alpha_1,\dots,\alpha_n$ are distinct algebraic numbers, then $\exp{\alpha_1},\dots,\exp{\alpha_n}$ are linearly independent over the algebraic numbers.
\end{theorem}

\noindent We will need the following consequence conveying the Lindemann-Weierstrass theorem for sine.

\begin{fact}\label{fact:LW-sine}
    If $\alpha_1,\dots,\alpha_n$ are nonzero algebraic numbers such that $|\alpha_i| \neq |\alpha_j|$ for each $i\neq j$, then $\sin \alpha_1,\dots, \sin \alpha_n$ are linearly independent over the algebraic numbers.
\end{fact}
\begin{proof}
    Suppose there are nonzero algebraic numbers $\alpha_1,\dots,\alpha_n$ and algebraic numbers $c_1,\dots,c_n$ not all zero such that
    \[ 0 = \sum\limits_{j=1}^n c_j \sin \alpha_j \] and $|\alpha_j| \neq |\alpha_j'|$ for each $j \neq j'$. Then
    \[ 0 = \sum\limits_{j=1}^n c_j \exp({\mathbbm{i}\alpha_j}) - \sum\limits_{j=1}^n c_j \exp({-\mathbbm{i}\alpha_j}) \]
    captures a linear dependence among $\exp({\mathbbm{i}\alpha_1}), \exp({-\mathbbm{i}\alpha_1}),\dots,\exp({\mathbbm{i}\alpha_n}), \exp({-\mathbbm{i}\alpha_n})$ over the algebraic numbers. By the assumption for nonzero and distinct absolute values, these are distinct algebraic numbers; this contradicts Lindemann-Weierstrass above.
\end{proof}

\noindent For complex numbers $\alpha_1,\dots,\alpha_n$, let $\td_{\Q}(\alpha_1,\dots,\alpha_n)$ denote the transcendence degree of $\alpha_1,\dots,\alpha_n$ over $\Q$ and let $\ldim_{\Q}(\alpha_1,\dots,\alpha_n)$ be the dimension of the $\Q$-linear subspace of $\C$ spanned by $\alpha_1,\dots,\alpha_n$. We now state the following equivalent variant of the famous conjecture due to Schanuel (first stated in Lang \cite{MR0214547}).

\begin{conj}[Schanuel's conjecture]\label{conj:schanuel}
Let $\alpha_1,\dots,\alpha_n \in \C$. Then
\[
\td_{\Q}(\alpha_1,\dots,\alpha_n,\exp{\alpha_1},\dots,\exp{\alpha_n})\geq \ldim_{\Q}(a_1,\dots,a_n).
\]
\end{conj}
\noindent We will need the following consequence conveying Schanuel's conjecture for sine.

\begin{fact}\label{fact:sinesc} Assume
Schanuel's conjecture. Let $\alpha_1,\dots,\alpha_n \in \C$. Then \[ \td_{\Q}(\alpha_1,\dots,\alpha_n, \sin \alpha_1, \dots, \sin \alpha_n) \geq \ldim_{\Q}(\alpha_1,\dots,\alpha_n). \]
\end{fact}
\begin{proof}
Let $\alpha_1,\dots,\alpha_n \in \C$.
For $z \in \C$, notice that \[ f^2(2 - 4g^2 - f^2) - 1 = 0 \] where $f = \exp({\mathbbm{i}z})$ and $g=\sin z$. So $\exp({\mathbbm{i}z})$ and $\sin z$ are algebraically dependent.
By Schanuel's conjecture, \begin{align*}
    \td_{\Q}(\alpha_1,\dots,\alpha_n,\sin \alpha_1, \dots, \sin \alpha_n)
    &= \td_{\Q}\big(\mathbbm{i}\alpha_1,\dots,\mathbbm{i}\alpha_n, \exp({\mathbbm{i}\alpha_1}),\dots, \exp({\mathbbm{i}\alpha_n})\big) \\
    &\geq \ldim_{\Q}(\mathbbm{i}\alpha_1,\dots,\mathbbm{i}\alpha_n) \\
    &= \ldim_{\Q}(\alpha_1,\dots,\alpha_n) .
    \qedhere
\end{align*}
\end{proof}

The next fact expresses how sine precipitates algebraic relations from linear ones.

\begin{fact}\label{fact:td_zero}
    Let $\alpha_1,\dots,\alpha_n \in \C$ and let $\alpha \in \C$ be a $\Q$-linear combination of $\alpha_1,\dots,\alpha_n$. Then $\sin \alpha$ is algebraic over $\sin \alpha_1,\dots, \sin \alpha_n$.
\end{fact}
\begin{proof}
    Let $c_1,\dots,c_n \in \Q$ be not all zero such that $\alpha = \sum_{j=1}^n c_j \alpha_j$. Then: \begin{align*}
        2\mathbbm{i} \sin \alpha 
        &= 2\mathbbm{i} \sin \Big( \sum_{j=1}^n c_j \alpha_j \Big) \\
        &= \exp\!\big({\mathbbm{i} \sum_{j=1}^n c_j \alpha_j}\Big) - \exp\!\big({-\mathbbm{i} \sum_{j=1}^n c_j \alpha_j}\Big) \\
        &= \prod\limits_{j=1}^n \exp({\mathbbm{i} c_j \alpha_j}) - \prod\limits_{j=1}^n \exp({-\mathbbm{i} c_j \alpha_j}) .
    \end{align*}
    
    For each $j$, since $c_j$ is rational, $\exp({\pm \mathbbm{i} c_j \alpha_j})$ is algebraically dependent over $\exp({\mathbbm{i} \alpha_j})$. Since $\exp({\mathbbm{i} \alpha_j})$ is algebraically dependent over $\sin \alpha_j$ by the proof of Fact \ref{fact:sinesc}, we have that $\exp({\pm \mathbbm{i} c_j \alpha_j})$ and $\sin \alpha_j$ are algebraically dependent.
    Thus $\sin \alpha$ is algebraic over $\sin \alpha_1,\dots, \sin \alpha_n$.
\end{proof}

\begin{defn}
Let $a/b$, $c/d \in \Q$ be such that $a,c \in \Z_{\neq 0}$, $b,d \in \Z_{>0}$, and $\gcd(a,b) = \gcd(c,d) = 1$.
The greatest common divisor of two rational numbers is \[ \gcd\Big(\frac{a}{b}, \frac{c}{d}\Big) := \frac{\gcd(a,c)}{\lcm(b,d)} \]
with $\gcd(a/b, 0) := a/b$ and $\gcd(0, c/d) := c/d$.
Now let $\alpha_1,\dots,\alpha_n \in \Q$. Recursively, their greatest common divisor is \[ \gcd(\alpha_1,\dots,\alpha_n) := \gcd\!\big(\gcd(\alpha_1,\dots,\alpha_{n-1}), \alpha_n \big). \]
Let $\beta_1,\dots,\beta_n \in \Z_{>0}$ and recursively define their least common multiple by
\[ \lcm( \beta_1,\dots,\beta_n ) := \lcm\big( \lcm( \beta_1,\dots,\beta_{n-1} ), \beta_n \big) \]
with $\lcm(\beta_1) := \beta_1$. Setting $B = \{ \beta_1,\dots,\beta_n \}$, we write $\lcm(B) := \lcm(\beta_1,\dots,\beta_n)$ and define $\lcm(\emptyset) := 1$.
\end{defn}

\begin{fact}\label{fact:gcd}
Let $\alpha_1,\dots,\alpha_n \in \Q$. The set of integer combinations of $\alpha_1,\dots,\alpha_n$ is exactly the set of all integer multiples of $\gcd(\alpha_1,\dots,\alpha_n)$.
\end{fact}
\begin{proof}
    Since the greatest common divisor always divides an integer combination, the forward direction is clear.
    So let $\beta$ be a multiple of $d := \gcd(\alpha_1,\dots,\alpha_n)$, say $\beta = md$ for some $m \in \Z$.
    By B\'{e}zout's lemma for principal ideal domains, let $k_1,\dots,k_n \in \Z$ be such that $d = \sum_{i=1}^n k_i \alpha_i$.
    Hence $\beta = \sum_{i=1}^n (mk_i)\alpha_i$
    with each $mk_i \in \Z$, so $\beta$ is an integer combination of $\alpha_1,\dots,\alpha_n$.
\end{proof}

\begin{fact}\label{fact:finiteset}
    Let $\alpha_1,\dots,\alpha_{n+1} \in \Q$ and  let $f : \Z^n \to \Q$ be the affine function
    \[ (x_1,\dots,x_n) \mapsto \Big( \sum\limits_{i=1}^n \alpha_i x_i \Big) + \alpha_{n+1} . \]
    Then the intersection of a bounded interval $I$ and the range of $f$ is the finite set \[ \{ k\gcd(\alpha_1,\dots,\alpha_n) + \alpha_{n+1} \in I : k \in \Z \} . \]
\end{fact}
\begin{proof}
    The range of the function $(x_1,\dots,x_n)\mapsto \sum_{i=1}^n \alpha_i x_i$ over the integers is the set of integer combinations of $\alpha_1,\dots,\alpha_n$.
    By Fact \ref{fact:gcd}, these are precisely the integer multiples of $\gcd(\alpha_1,\dots,\alpha_n)$.
    Adding the constant $\alpha_{n+1}$ simply shifts the range.
\end{proof}

\subsection{Decidability results}
In \cite{MR1435773}, Macintyre and Wilkie showed that under \linebreak Schanuel's conjecture $\FO(\R,<,+,\cdot,\exp)$ is decidable, conditionally solving Tarski's exponential function problem. While not explicitly stated in \cite{MR1435773}, expanding this structure by restricted sine preserves decidability of the theory.

\begin{fact}[Macintyre and Wilkie, see Theorem 3.1 of \cite{MR3522649}]\label{fact:mw}
Assume Schanuel's conjecture holds. Then $\FO(\R,<,+,\cdot,\exp,\sin|_{[0,n]})$ is decidable.
\end{fact}

\noindent It is well-known that Fact \ref{fact:mw} fails when restricted sine is replaced by unrestricted sine. However, here we use Fact \ref{fact:mw} to show that the first-order theory of the ordered additive group of real numbers with unrestricted sine is decidable. 

\begin{thm}\label{thm:decide-additive-gp}  
Assume Schanuel's conjecture holds. Then  $\FO(\R,<,+,\sin)$ is decidable.
\end{thm}

\noindent By Toffalori and Vozoris \cite[Theorem 2.7]{MR2582162} the structure $(\R,<,+,\sin)$ is locally o-minimal. This implies that every subset of $\R$ definable with parameters in this structure is a union of an open set and a discrete set. Their construction was later generalized in Kawakami et al. \cite[Theorem 25]{KTTT}.
Although it is not hard to see how their technology can be adjusted to prove Theorem \ref{thm:decide-additive-gp}, unfortunately neither paper gives a proof or mention of this result. Therefore, we include a detailed description here of how to derive Theorem \ref{thm:decide-additive-gp} from Fact \ref{fact:mw} using the idea from \cite{MR2582162,KTTT}.

\subsubsection{Simple products} We first introduce simple products, roughly in the same way they were used in \cite{KTTT}. A similar construction also appeared in Bouchy, Finkel, and Leroux \cite[Section 2.3]{BFL}.
\newline

\noindent Let $\mathcal{M}_1$ be an $\mathcal{L}_1$-structure and $\mathcal{M}_2$ be an $\mathcal{L}_2$-structure for signatures $\mathcal{L}_1,\mathcal{L}_2$.
For $A\subseteq M_1^n$ and $B\subseteq M_2^n$, we define $A\ast B$ to be the subset of $M_1^n \times M_2^n$ given by
\[
\Big\{ \big((a_1,b_1),\dots,(a_n,b_n)\big) \in (M_1 \times M_2)^n \ : (a_1,\dots,a_n) \in A, (b_1,\dots,b_n) \in B\Big\}
. \]
Let $\mathcal{L}'$ be the signature consisting of $n$-ary relation symbols $R_{A,B}$ for every $n\in \N$ and every pair $(A,B)$ of an $\mathcal{L}_1$-definable subset $A$ of $M_1^n$ and an $\mathcal{L}_2$-definable subset $B$ of $M_2^n$. 
Define $\mathcal{M}_1 \ast \mathcal{M}_2$ to be the $\mathcal{L}'$-structure on $M_1 \times M_2$ in which each predicate $R_{A,B}$ is interpreted as $A\ast B$. We call $\mathcal{M}_1\ast \mathcal{M}_2$ the \textbf{simple product} of $\mathcal{M}_1$ and $\mathcal{M}_2$.\newline  

\noindent Let $A_1,A_2\subseteq M_1^{n}$ and $B_1,B_2\subseteq M_2^{n}$. Observe that
\begin{align}
(A_1\ast B_1) \setminus (A_2 \ast B_2) &= \big(A_1\setminus A_2) \ast B_1\big) \cup \big(A_1 \ast (B_1 \setminus B_2)\big),\label{eq:ast1}\\
(A_1 \ast B_1) \cap (A_2 \ast B_2) &= (A_1\cap A_2)\ast (B_1 \cap B_2),  \text{ and }\label{eq:ast2}\\
\pi(A_1 \ast B_1) &= \pi(A_1) \ast \pi(B_1),\label{eq:ast3}
\end{align}
where $\pi$ denotes the projection of any $n$-fold Cartesian product onto the first $n-1$ coordinates. 

\begin{fact}\label{fact:simproddef}
Let $X\subseteq (M_1\times M_2)^n$ be definable in $\mathcal{M}_1\ast \mathcal{M}_2$. Then there are $A_1,\dots,A_m\subseteq M_1^n$ definable in $\mathcal{M}_1$ and $B_1,\dots,B_m \subseteq M_2^n$ definable in $\mathcal{M}_2$ such that $X = \bigcup_{i=1}^m A_i \ast B_i$. Moreover, given the $\mathcal{L}'$-formula defining $X$ we can compute $m$ as well as the $\mathcal{L}_1$- and $\mathcal{L}_2$-formulas defining $A_1,\dots,A_m$ and $B_1,\dots,B_m$. 
\end{fact}
\begin{proof}
By the observation above, we can apply equations \eqref{eq:ast1}, \eqref{eq:ast2} and \eqref{eq:ast3} finitely many times to transform $X$ into a finite union of sets with form $A\ast B$ for $A\subseteq M_1^n$ definable in $\cal{M}_1$ and $B\subseteq M_2^n$ definable in $\cal{M}_2$.
This procedure is clearly effective.
\end{proof}

\begin{cor}\label{cor:proddec}
If $\FO(\mathcal{M}_1)$ and $\FO(\mathcal{M}_2)$ are decidable, then so is $\FO(\mathcal{M}_1\ast \mathcal{M}_2)$.
\end{cor}
\begin{proof}
Let $\sigma$ be an $\mathcal{L}'$-sentence. We can assume that $\sigma$ is of the form $\exists x \varphi(x)$ for some $\mathcal{L}'$-formula $\varphi$ with free variable $x$. By Fact \ref{fact:simproddef} and \eqref{eq:ast3}, we can compute $m \in \N$, $\cal{L}_1$-formulas $\psi_1,\dots,\psi_m$, and $\cal{L}_2$-formulas $\chi_1,\dots,\chi_m$ such that
\[
\mathcal{M}_1\ast \mathcal{M}_2 \models \exists x \varphi \hbox{ iff there is $i\in \{1,\dots,m\}$ such that } \mathcal{M}_1\models \exists x \psi_i \hbox{ and } \mathcal{M}_2\models \exists x \chi_i.
\]
Since the first-order theories of $\cal{M}_1$ and $\cal{M}_2$ are decidable, we can decide the truth of the right hand side.
\end{proof}

\begin{proof}[Proof of Theorem \ref{thm:decide-additive-gp}]
Let $\mathbb{D}$ be the interval $[-\pi,\pi)$.
Notice $\sin|_{\mathbb{D}}$ is definable in $(\R,<,+,\cdot,\sin|_{[0,4]})$, so $\FO\left(\R, <, +, \sin|_{\mathbb{D}}\right)$ is decidable by Fact \ref{fact:mw}.
Define $+_{\pi} : \mathbb{D}^2 \to \mathbb{D}$ to be the function mapping $a,b \in \mathbb{D}$ to 
\[
a+_{\pi} b := \begin{cases}
a+b+2\pi & a+b < -\pi\\
a+b & a+b \in \mathbb{D}\\
a+b-2\pi & a+b \geq \pi .
\end{cases}
\]
For $u\in \{-1,0,1\}$, we define $X_i$ to be the subset of $\mathbb{D}$ such that 
\[
X_u := \begin{cases}
\big\{(a,b) \in \mathbb{D}^2 \ : \  a+b < -\pi\big\} & u=-1\\
\big\{(a,b) \in \mathbb{D}^2 \ : \  a+b \in \mathbb{D}\big\} & u=0\\
\big\{(a,b) \in \mathbb{D}^2 \ : \  a+b \geq \pi \big\} & u=1.
\end{cases}
\]
Now notice that the structure $\mathcal{D}:=(\mathbb{D},<,+_{\pi},\sin|_{\mathbb{D}},X_{-1},X_0,X_1)$ is definable in $\left(\R, <, +, \sin|_{\mathbb{D}}\right)$. Hence $\FO(\mathcal{D})$ is decidable. Let $\mathcal{Z}:=(\Z,<,+)$, whose first-order theory is well-known to be decidable. By Corollary \ref{cor:proddec}, the theory $\FO(\mathcal{D}\ast \mathcal{Z})$ is decidable.\newline

\noindent We complete the proof by defining an isomorphic copy of $\left(\R,<,+,\sin\right)$ in $\mathcal{D} \ast \mathcal{Z}$. Consider the linear order $\prec$ on $\mathbb{D}\times \Z$ defined such that for all $(a,k),(b,k') \in \mathbb{D} \times \Z$, we have that $(a,k)\prec (b,k')$ if and only if either $k<k'$, or, $k=k'$ and $a<b$. This order is definable in $\cal{D}\times \cal{Z}$ since $\prec$ as a subset of $(\mathbb{D}\times \Z)^2$ is equal to
\[
 \big(\mathbb{D}^2 \ast \big\{ (k,k') \in \Z^2 \ : \ k<k'\big\}\big) \cup \big(\big\{(a,b) \in \mathbb{D}^2 \ : \ a<b \big\} \ast \Z^2 \big)
. \]
Define $\widetilde{+}: (\mathbb{D} \times \Z)^2\to \mathbb{D}\times \Z$ to be the function mapping $(a,k),(b,k') \in \mathbb{D} \times \Z$ to
\begin{align*}
(a, k)\ \widetilde{+}\ (b, k') &:= \begin{cases}
(a+b+2\pi,\ k+k'-1) & a+b < -\pi \\ (a+b,\ k+k') & a+b \in \mathbb{D} \\ (a+b-2\pi,\ k+k'+1) & a+b \geq \pi .
\end{cases} 
\end{align*}
Note that the graph of $\widetilde{+}$ is 
\[
\bigcup_{u \in \{-1,0,1\}}
\big\{ (k,k',\ell) \in \Z^3 \, : \, k+k'+u=\ell \big\} \ast 
\big\{ (a,b,c) \in \mathbb{D}^3 \, : \, a +_\pi b = c,\ (a,b)\in X_u \big\}    
\]
and hence definable in $\mathcal{D} \ast \mathcal{Z}$. Let $\widetilde{\sin} : \mathbb{D} \times \Z \to \mathbb{D}\times \Z$ map $(x,k)$ to $(\sin x,0)$. The graph of $\widetilde{\sin}$ is just $\{(x,\sin x) \, : \, x \in \mathbb{D}\} \ast \{(k,0) \, : \, k \in \Z \}$ and hence definable in $\mathcal{D} \ast \mathcal{Z}$. Now observe that $\mu : \mathbb{D} \times \Z \to \R$ defined by $(a,k) \mapsto a + 2\pi k$ is an isomorphism between 
$(\mathbb{D}\times \Z,\prec ,\widetilde{+},\widetilde{\sin})$ and $(\R,<,+,\sin)$.
\end{proof}

\section{Upper bound for decidability}

In this section, we present the proof of Theorem A. We follow the main line of reasoning from the proof of \cite[Theorem 7.1]{alphaPA}, which in turn is based on the ideas from \cite{MR3223381}. However, all arguments based on Ostrowski representation have to be reframed in the absence of irrational scalar multiplication and in the presence of sine. By carefully analyzing the original work in \cite{MR3223381}, we are able to obtain comparable bounds in spite of lacking any comparable numeration system derived from the sine function. Throughout this section, we abuse notation by using $\N$ rather than $\Z$ when working with $\sinPA$. An easy exercise verifies that the construction with $\Z$ defined in Section \ref{sec:formalize-sinPA} is interdefinable with the analogous construction quantifying over $\N$ without additional quantifiers, so the conclusion of Theorem A is unaffected.\newline

\noindent We will produce a 6-ary $\sin$-PA formula \Member such that for every finite set $S \subseteq \N^2$ there is ${\bf X} \in \N^4$ satisfying \[ (s,t) \in S \iff \Member({\bf X}, s, t) \]
for all $(s, t) \in \N^2$. This is sufficient to show that the set of all $\sin$-PA sentences is undecidable. Indeed, the weak monadic second-order theory of the grid $(\N^2,s_1,s_2)$, where $s_1(m,n):=(m+1,n)$ and $s_2(m,n):=(m,n+1)$, is well-known to be undecidable. Using $\Member$, we can reduce the decision problem of this theory to that of $\sin$-PA sentences.
By using the following result, which is implicit in \cite[proof of Theorem 7.1]{alphaPA}, we obtain a concrete bound on the necessary quantifier alternations and the size of quantifier blocks.

\begin{fact}\label{fact:hpn} Let $\mathcal{N}$ be a first-order expansion of $(\N,<,+)$ by a $6$-ary predicate $P$ such that for every finite set $S \subseteq \N^2$ there is ${\bf X} \in \N^4$ satisfying
\[ (s,t) \in S \iff ({\bf X}, s, t)\in P \]
for all $(s, t) \in \N^2$.
Then the truth in $\mathcal{N}$ of $\exists^{50}\forall^3$-sentences containing at most 242 appearances of $P$ is undecidable.
\end{fact}

\noindent We will evince an $\exists^{10}\forall^{14}$ $\sin$-PA formula $\Member$ with the desired property. The bound in Theorem A then follows from that in Fact \ref{fact:hpn}.

\subsection{Constructing \Member}
We begin with the left-approximates of natural numbers via the sine function.

\begin{defn}\label{def:approx}
Let $d,X \in \N$. The {\bf best approximate} of $X$ up to $d$, denoted $X|_d$, is the number in $\N_{\leq d}$ that best approximates $X$ from the left under the sine function if it exists\footnote{If there is no number in $\N_{\leq d}$ whose sine value is less than $\sin X$, we say $X|_d$ is undefined. In practice, this is not an issue since we ultimately work with numbers with positive sine values.}; that is, $X|_d = Y$ if $Y \in \N$, $Y \leq d$, and $\sin Y$ is the maximum value from $\sin\left(\N_{\leq d}\right) \cap (-\infty, \sin X]$.

We refer to $d$ as the degree of approximation.
Further if $Y$ is the best approximate of $X$ up to some $d$, we simply write that $Y$ is a best approximate of $X$ without mention of the degree.
\end{defn}

\noindent We observe that the relation $X|_d = Y$ is definable by a $\forall^1$ $\sin$-PA formula: \begin{align*}
 {\bf Better}(d,X,Y,Z) :=\ &Y \leq d \wedge \sin Y \leq \sin X \\ &\wedge (Z \leq d \wedge Z \neq Y \wedge \sin Z \leq \sin X) \rightarrow \sin Z < \sin Y, \\
 X|_d = Y :=\ &\forall Z\ {\bf Better}(d,X,Y,Z).
\end{align*}

\begin{lem}\label{lem:bar_d_con}
Let $X, d \in \N$ be such that $X \leq d$. Then there is an interval $I \subseteq [-1, 1]$ containing $\sin X$ such that for all $Y \in \N$, \[\sin Y \in I \implies Y|_d = X \text{ and } \sin X \leq \sin Y .\]
\end{lem}
\begin{proof}

Set $b := \min(\! \left\{\beta \in \sin\left(\N_{\leq d}\right) : \sin X < \beta \right\})$ if the set is nonempty; otherwise, let $b = 1$. Set $I := \left[ \sin X,\, b\right)$. Now let $Y \in \N$ such that $\sin Y \in I$.
Then $\sin X \leq \sin Y$, and there is no $Z \in \N_{\leq d} \setminus \{X\}$ such that $\sin X \leq \sin Z < \sin Y$, by choice of $b$. Hence $Y|_d = X$.
\end{proof}

\begin{lem}\label{lem:bar_d_hyp}
Let $X \in \N$ and $J \subseteq (-1, 1)$ be an open interval around $\sin X$. Then there is arbitrarily large $d \in \N$ such that for all $Y \in \N$, \[ Y|_d = X \implies \sin Y \in J . \]
\end{lem}
\begin{proof}
Let $b \in \N$.
Take $d \in \N$ to be minimal such that $d > b$ and
\[ \sin d \in \left(\sin X, \min\!\big( \{ \sup J\} \cup \{ \sin Z : Z \leq b \text{ and } \sin X < \sin Z \} \big) \right) . \]
Thus $\left( \sin X,\, \sin d \right) \subseteq J$ and $d \neq X$.
Now suppose $Y \in \N$ such that $Y|_d = X$; that is, $X$ is the best approximate of $Y$ up to $d$. Then $\sin X < \sin Y < \sin d < \sup J$, and we have $\sin Y \in J$ as desired.
\end{proof}

\begin{lem}
  Let $X,d,d' \in \N$ be such that $d < d'$.
  Then $X|_d = \big(X|_{d'}\big)|_d$.
\end{lem}
\begin{proof}
    Let $Y \in \N_{\leq d}$ be such that $X|_d = Y$. So
    \[ \sin Y = \max \big( \{ \sin Z : Z \in \N_{\leq d} \text{ and } \sin Z \leq \sin X \} \big) . \]
    Further let $Y' \in \N_{\leq d'}$ be such that $X|_{d'} = Y'$. So
    \[ \sin Y' = \max \big( \{ \sin Z : Z \in \N_{\leq d'} \text{ and } \sin Z \leq \sin X \} \big) . \]
    Finally let $Y'' \in \N_{\leq d}$ be such that $\big(X|_{d'}\big)|_d = Y''$. So
    \[ \sin Y'' = \max \big( \{ \sin Z : Z \in \N_{\leq d} \text{ and } \sin Z \leq \sin Y' \} \big) . \]
    Thus since $d < d'$, we have $\sin Y \leq \sin Y' \leq \sin X$.
    By construction of $Y$, there is no $Z \in \N_{\leq d} \setminus \{ Y \}$ such that $\sin Y \leq \sin Z \leq \sin X$.
    Hence $Y'' = Y$.
\end{proof}

\noindent Hereafter, take $\overline{X} = (X_1, X_2)$, $\overline{Y} = (Y_1, Y_2)$, and $\overline{Z} = (Z_1, Z_2)$. We say tuples are equal (distinct) as sets when the respective sets of elements from the tuples are equal (distinct).

\begin{lem}\label{lem:sine-diff}
Let $\overline{X}, \overline{Y} \in \N^2$. If \[ |\sin X_1 - \sin X_2 | = |\sin Y_1 - \sin Y_2 | \neq 0, \] then $\overline{X}$ and $\overline{Y}$ are equal as sets.
\end{lem}
\begin{proof}
Suppose $|\sin X_1 - \sin X_2 | = |\sin Y_1 - \sin Y_2 | \neq 0$. Then
\begin{align*}
    0 &\neq \big| \exp({\mathbbm{i}X_1}) - \exp({-\mathbbm{i}X_1}) - \exp({\mathbbm{i}X_2}) + \exp({-\mathbbm{i}X_2})\big| \\
    &= \big| \exp({\mathbbm{i}Y_1}) - \exp({-\mathbbm{i}Y_1}) - \exp({\mathbbm{i}Y_2}) + \exp({-iY_2})\big| .
\end{align*}
Repeated applications of the Lindemann-Weierstrass theorem (Fact \ref{thm:baker}) yield that $\{X_1,-X_1,X_2,-X_2\}$ and $\{Y_1,-Y_1,Y_2,-Y_2\}$ are the same set. Since $\overline{X},\overline{Y} \in \N^2$, we further have that $\{X_1,X_2\}$ and $\{Y_1,Y_2\}$ are the same set.
\end{proof}

\begin{defn}
Define $g: \N^4 \to \R$ as the function that maps $\left(\overline{X}, \overline{Y}\right)$ to
\[\big| \sin X_2 - \sin X_1 - \left|\sin Y_2 - \sin Y_1\right| \big|.\]
\end{defn}

\begin{lem}
Let $\overline{X}, \overline{Y} \in \N^2$. Then $g\left(\overline{X}, \overline{Y}\right)=0$ if and only if either $\sin X_1 < \sin X_2$ and $\overline{X}$ and $\overline{Y}$ are equal as sets, or, $X_1 = X_2$ and $Y_1 = Y_2$.
\end{lem}
\begin{proof}
By definition of $g$, we know that $g\left(\overline{X}, \overline{Y}\right)=0$ if and only if $\sin X_2 - \sin X_1 = |\sin Y_1 - \sin Y_2|$. If these differences are equal to 0, then $X_1 = X_2$ and $Y_1 = Y_2$. Otherwise $\sin X_1 < \sin X_2$, and Lemma \ref{lem:sine-diff} completes the proof.
\end{proof}

\begin{defn}
Let {\bf Best} be the relation on $\N \times \N \times \N^2 \times \N$ that holds precisely for all tuples $\left(d, e, \overline{X}, Y_1\right)$ for which there exists $Y_2 \in \N$ such that the following hold:
\begin{enumerate}[(i)]
 \item $Y_1 \leq d$, $Y_2 \leq e$, $Y_1 < Y_2$,
 \item $g(\overline{X}, \overline{Y}) < g(\overline{X}, \overline{Z})$ for all $\overline{Z} \in \N_{\leq d} \times \N_{\leq e}$ with $\overline{Z}$ and $\overline{Y}$ distinct as sets.
\end{enumerate}
\end{defn}

{\bf Best} should be understood as the analogue of a best approximate under sine to a best approximate under {\sl the difference of sines} as captured by $g$.
Indeed, {\bf Best} holds exactly when $Y_1$ is the lesser of such a best difference approximate of $X_1$ and $X_2$; the superlative name is appropriate since given $d,e\in \N$ and distinct $X_1, X_2 \in \N$, there is a unique pair $\overline{Y}$ minimizing $g(\overline{X},\cdot)$ over $\N_{\leq d} \times \N_{\leq e}$. Hence there is at most one $Y_1 \in \N_{\leq d}$ such that ${\bf Best}(d,e,\overline{X},Y_1)$ holds.

\begin{lem}\label{lem:best_def}
{\bf Best} is definable by an $\exists^1\forall^2$ $\sin$-PA formula.
\end{lem}
\begin{proof}
Observe that ${\bf Best}(d, e, \overline{X}, Y_1)$ holds if and only if
\begin{align*}
 &\exists Y_2 \ \forall Z_1,Z_2 \ \ 
 Y_1 \leq d \wedge Y_2 \leq e \wedge Y_1 < Y_2 \\
 & \wedge \big( Z_1 \leq d \wedge Z_2 \leq e \wedge (Z_1 = Y_1 \rightarrow Z_2 \neq Y_2) \wedge (Z_1 = Y_2 \rightarrow Z_2 \neq Y_1) \big) \rightarrow \\
 &\hspace{1cm} \big| \sin X_2 - \sin X_1 - |\sin Y_2 - \sin Y_1|\big| < \big| \sin X_2 - \sin X_1 - |\sin Z_2 - \sin Z_1|\big|.
\end{align*}
Observe that $-$ and $|\cdot |$ are quantifier-free $\sin$-PA definable.
\end{proof}

\begin{figure}
\includegraphics[scale=0.405]{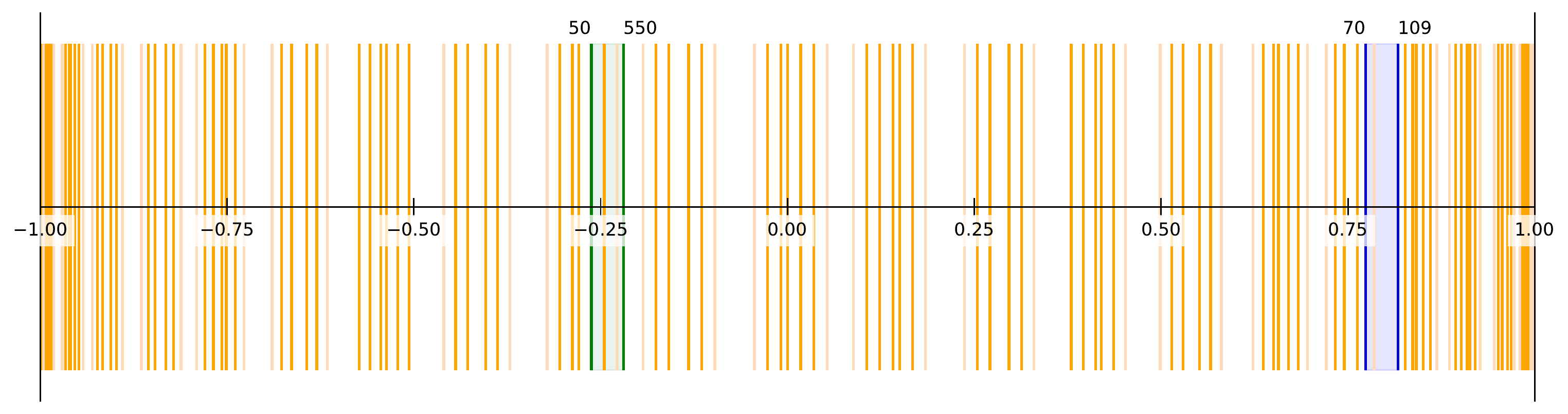}
\caption{Visualization of the witness to ${\bf Best}(100,140,50,550,70)$.}\label{fig:diffmesh}
\end{figure}
\begin{example}
Figure \ref{fig:diffmesh} illustrates that ${\bf Best}(100,140,50,550,70)$ is witnessed by $Y_2 = 109$. The dark orange vertical bars mark the elements of $\sin(\N_{\leq 100})$ while the light orange bars fill in the rest of $\sin(\N_{\leq 140})$. The difference between $\sin(50)$ and $\sin(550)$ is highlighted in green. Of all possible pairs of bars---distinct from $\{\sin(50), \sin(550)\}$ and with at least one being dark orange---the blue-shaded difference between $\sin(70)$ and $\sin(109)$ is closest in value to that between $\sin(50)$ and $\sin(550)$.
\end{example}

\begin{defn}
Define {\bf Next} to be the relation on $\N^2$ that holds precisely for all tuples $(X_1, X_2) \in \N^2$ such that $X_1 < X_2$ are consecutive best approximates of some natural number.
\end{defn}

\begin{lem}\label{lem:next}
Let $X_1 \in \N$. Then there exists arbitrarily large $X_2 \in \N$ such that ${\bf Next}(X_1,X_2)$ holds.
\end{lem}
\begin{proof}
Let $b \in \N$ such that $b > X_1$.
By Lemma \ref{lem:bar_d_con}, there is an interval $I$ containing $\sin X_1$ such that for all $Y\in \N$, the membership $\sin Y \in I$ implies $Y|_b = X_1$.
Since $\sin(\N)$ is dense in $[-1,1]$, pick the least $X_2 \in \N$ such that $X_2 \geq b$ and $\sin X_2 \in I$. Thus $X_2|_{X_2 - 1} = X_1$. 
Then $X_2|_{X_2} = X_2$ and $X_2|_{X_1} = X_1$, so $X_1$ and $X_2$ satisfy ${\bf Next}(X_1,X_2)$ as desired.
\end{proof}

\begin{lem}\label{lem:next_seq}
Let $k_1,\dots,k_n \in \N$ be such that ${\bf Next}(k_i,k_{i+1})$ holds for each $i=1,\dots,n-1$.
If $d \in \N$ satisfies $k_i \leq d < k_{i+1}$ for some $i=1,\dots,n-1$, then $k_n|_{d} = k_i$.
Further if $k_1 = 0$, then every best approximate of $k_n$ is equal to $k_i$ for some $i \in \{1,\dots,n\}$.
\end{lem}
\begin{proof}
    Let $d \in \N$ be such that $k_i \leq d < k_{i+1}$.
    Since ${\bf Next}(k_i,k_{i+1})$ holds, notice that $k_i < k_{i+1}$ and further $k_{i+1}|_{d'} = k_i$ holds for each $d' \in \N$ such that $k_i \leq d' < k_{i+1}$. Then since $k_i \leq d < k_{i+1}$, we have $k_{i+1}|_{d} = k_i$.
    So if $i = n-1$, we are done.
    
    Otherwise, let $j = \{i+2,\dots,n\}$ and suppose $k_{j-1}|_{d} = k_i$.
    Consider $k_{j}|_d$. Note $k_{j}|_{k_{j-1}} = k_{j-1}$ since ${\bf Next}(k_{j-1},k_{j})$ holds.
    Then since $d < k_{i+1} \leq k_{j-1}$, we have $k_{j}|_{d} = \big(k_{j}|_{k_{j-1}}\big)|_{d} = k_{j-1}|_{d} = k_i$.
    By induction on $j$, the first statement of the lemma follows.
    
    Now suppose $k_1 = 0$. Let $d \in \N$ be a best approximate of $k_{n}$.
    If $d \geq k_n$, then clearly $d = k_n$. Otherwise since \[ \N_{< k_n} = \bigcup_{i=1}^{n-1} \{ k_i,k_i + 1,\dots,k_{i+1}-1\} , \]
    the first statement completes the proof.
\end{proof}

\noindent Notice that ${\bf Next}(X_1, X_2)$ is $\forall^1$-definable by the following $\sin$-PA formula:
\[X_1 < X_2 \wedge \sin X_1 < \sin X_2 \wedge \forall Z (Z < X_2) \rightarrow (\sin Z \leq \sin X_1 \vee \sin X_2 < \sin Z).\]

\noindent With Lemma \ref{lem:best} below, we will see that for given $d \in \N$, we may pick $e \in \N$ sufficiently large that for all $X_1 \in \N$ and $Y_1 \leq d$, the set \[ \left\{ X_2 \in \N : {\bf Best}\left(d, e, X_1, X_2, Y_1\right)\right\} \] is cofinal in $\N$. 
Lemma \ref{lem:best} is crucial in what follows and should be compared to condition (ii) of Theorem A from \cite{MR3223381}.

\begin{lem}\label{lem:best}
Let $b, d, e_0, s \in \N$ and $\overline{X} \in \N^2$ be such that $\sin X_1 < \sin X_2$ and $s \leq d$. Then there exist $e \in \N$ and an interval $J \subseteq \left(\sin X_1, \, \sin X_2\right)$ such that $e \geq e_0$, ${\bf Next}(b,e)$ holds, and for all $Z \in \N$,
\[\sin Z \in J \implies {\bf Best}(d, e, X_1, Z, s).\]
\end{lem}
\begin{proof}
Since $\N_{\leq d}$ is finite and $\sin (\N)$ is dense in $[-1,1]$, there is $e \in \N$ such that 
\begin{enumerate}
    \item $e> \max\{ e_0, d\}$ and  
    \item for all $w_1 \in \N_{\leq d}$, there exists $w_2 \in \N_{\leq e}$ with $w_1 < w_2$ and \[ |\sin w_2 - \sin w_1 | < \sin X_2 - \sin X_1 . \]
\end{enumerate}
Note that if $e$ satisfies (1) and (2), then so does every $e' \in \N_{\geq e}$. 
By Lemma \ref{lem:next}, we can thus find an $e\in \N$ such that (1), (2), and ${\bf Next}(b,e)$ hold. Fix this $e$. Since $s \leq d$, also fix $w \in \N_{\leq e}$ such that $s < w$ and $|\sin w - \sin s| < \sin X_2 - \sin X_1$.\newline

\noindent By the contrapositive of Lemma \ref{lem:sine-diff} and since $\N_{\leq d} \times \N_{\leq e}$ is finite, we may pick $\varepsilon > 0$ such that for all $(w_1, w_2) \in \N_{\leq d} \times \N_{\leq e}$ with $\{ w_1, w_2 \}$ and $\{ s, w \}$ distinct sets,
\[ \big| |\sin w_2 - \sin w_1| - |\sin w - \sin s|\big| > \varepsilon. \]
Set
\[ \delta := \sin X_1 + |\sin w - \sin s| \text { and } J := \Big(\delta - \frac{\varepsilon}{2}, \delta + \frac{\varepsilon}{2}\Big).\]
Since $|\sin w - \sin s| < \sin X_2 - \sin X_1$, we have that $J \subseteq (\sin X_1, \sin X_2)$. We now show the desired implication.\newline

\noindent Let $Z \in \N$ be such that $\sin Z \in J$. We need to show that ${\bf Best}(d, e, X_1, Z, s)$ holds. Indeed, for all $(w_1, w_2) \in \N_{\leq d} \times \N_{\leq e}$ with $\{ w_1,w_2 \}$ distinct from $\{ s, w \}$,
\begin{align*}
 g(X_1, Z, w_1, w_2) &= \big| \sin Z - \sin X_1 - |\sin w_2 - \sin w_1|\big| \\
 &= \big| \sin Z - \delta + |\sin w - \sin s| - | \sin w_2 - \sin w_1|\big| \\
 &\geq \Big| |\sin Z - \delta| - \big| |\sin w - \sin s| - |\sin w_2 - \sin w_1|\big|\Big| > \frac{\varepsilon}{2}
\end{align*}
since $|\sin Z - \delta| < \varepsilon/2$ and $\big| |\sin w - \sin s| - |\sin w_2 - \sin w_1|\big| > \varepsilon$. Moreover,
\[ g(X_1, Z, s, w) = \big| \sin Z - \sin X_1 - |\sin w - \sin s|\big| \leq |\sin Z - \delta| < \frac{\varepsilon}{2}. \]
So since $s < w$, we have that ${\bf Best}(d, e, X_1, Z, s)$ holds as desired.
\end{proof}

\begin{lem}\label{lem:bar_d_best}
Let $b, d, s \in \N$ and $\overline{X} \in \N^2$ be such that $s \leq d$ and $\sin X_1 < \sin X_2$.
Then there exist $e_1, e_2, Y \in \N$ such that the following hold:
\begin{enumerate}[(i)]
 \item $\sin X_1 < \sin Y < \sin X_2$,
 \item $d < e_1 < e_2$ and $Y \leq e_2$,
 \item ${\bf Next}(b,e_1)$,
 \item for all $Z \in \N$, \[ Z|_{e_2} = Y \implies {\bf Best}(d, e_1, X_1, Z, s).\]
\end{enumerate}
\end{lem}
\begin{proof}
By Lemma \ref{lem:best}, there is an open interval $J \subseteq \big(\sin X_1,\, \sin X_2\big)$ and $e_1 \in \N$ with $e_1 > d$ such that ${\bf Next}(b,e_1)$ holds and for all $Z \in \N$, \[ \sin Z \in J \implies {\bf Best}(d, e_1, X_1, Z, s).\]
Let $Y\in \N$ be such that $\sin Y \in J$. By Lemma \ref{lem:bar_d_hyp}, pick $e_2 \in \N$ at least as large as $Y$ and such that for all $Z \in \N$, \[ Z|_{e_2} = Y \implies \sin Z \in J.\] Then ${\bf Best}(d, e_1, X_1, Z, s)$ holds for all $Z \in \N$ with $Z|_{e_2} = Y$.
\end{proof}

\begin{defn}
Define $\Admissible$ to be the $12$-ary relation on $\N$ that holds precisely for all tuples $(d_1, d_2, d_3, d_4, e_1, e_2, X_1, X_2, X_3, X_4, s, t) \in \N^{12}$ such that the following hold: \begin{enumerate}[(i)]
 \item $d_1 < d_2 < d_3$ are consecutive best approximates of $X_1$,
 \item $d_1 \leq d_4 < d_2$ and $d_4$ is a best approximate of $X_3$,
 \item $d_1 \leq e_1 < d_2 \leq e_2 < d_3$ and $e_1,e_2$ are best approximates of $X_2$,
 \item ${\bf Best}(d_1, e_1, X_4|_{d_1}, X_4, s)$,
 \item ${\bf Best}(d_2, e_2, X_4|_{d_2}, X_4, t)$.
\end{enumerate}
Define \Member to be the $6$-ary relation on $\N$ that holds precisely for all tuples $(X_1, X_2, X_3, X_4, s, t) \in \N^6$ such that there exist $d_1, d_2, d_3, d_4, e_1, e_2 \in \N$ with $\Admissible(d_1,d_2,d_3,d_4,e_1,e_2,X_1,X_2,X_3,X_4,s,t)$.
\end{defn}

Now, we construct sequences whose final terms encode an arbitrary finite subset of $\N^2$ in a way that \Member can decode.
Compared to \cite[proof of Theorem 7.11]{alphaPA}, Lemma \ref{lem:member_seq} involves more work to recreate some properties which were natural consequences of the Ostrowski numeration systems employed throughout \cite{alphaPA}.
In particular, we construct an additional sequence $(m_j)_{j=0}^n$ to encode the parity of indices for the elements $c_i$ so that odd indices $i$ signal the beginning of pairs $(c_i,c_{i+1})$.
We further constrain the first three sequences using ${\bf Next}$ so that each may be encoded by a single number: its final term.

\begin{lem}\label{lem:member_seq}
Let $S \subseteq \N^2$ be finite and $c_1, \dots, c_{2n} \in \N$ be such that $S = \big\{(c_1,c_2), \linebreak\dots, (c_{2n-1},c_{2n})\big\}$. Then there exist three strictly increasing sequences $(k_i)_{i=0}^{2n}$,\linebreak $(\ell_i)_{i=0}^{2n}$, and $(m_j)_{j=0}^n$ of nonconsecutive natural numbers and a sequence $(W_i)_{i=0}^{2n}$ of natural numbers such that the following hold for $i = 0,\dots,2n$:
\begin{enumerate}
 \item[(1)] $W_i|_{k_j} = W_j$ for all $j\leq i$,
 \item[(2)] $k_i > \max\{c_1, \dots, c_{2n}\}$,
\end{enumerate}
and if $i \geq 1$, then:
\begin{enumerate}
 \item[(3)] $k_{i-1} < \ell_i < k_i$,
 \item[(4)] $k_i > W_i$ and for all $Z \in \N$, \[ Z|_{k_i} = W_i \implies {\bf Best}(k_{i-1}, \ell_i, W_{i-1}, Z, c_i), \]
 \item[(5)] ${\bf Next}\left(k_{i-1}, k_i\right)$,
 \item[(6)] ${\bf Next}\left(\ell_{i-1}, \ell_i\right)$,
\end{enumerate}
and if $i$ is odd, then:
\begin{enumerate}
 \item[(7)] $k_{i-1} \leq m_{(i-1)/2} < k_i$
 \item[(8)] ${\bf Next}\left(m_{(i-1)/2 - 1}, m_{(i-1)/2}\right)$ for $i > 1$.
\end{enumerate}
\end{lem}
\begin{proof}
We construct these sequences recursively. For the base case, we pick $k_0 \in \N$ such that $k_0 > \max\{c_1,\dots,c_{2n}\}$ and ${\bf Next}(0, k_0)$ by Lemma \ref{lem:next}.
Set $m_0 = W_0 = k_0$ and $\ell_0 = 0$.

Now, let $i \geq 1$ and suppose that we have constructed $k_0,\dots,k_{i-1}$, $\ell_1, \dots, \ell_{i-1}$, $m_0, \dots, m_{\lceil(i-1)/2\rceil-1}$, and $W_0, \dots, W_{i-1}$ such that the above conditions $(1) - (8)$ hold for $j = 0,\dots, i-1$. We now seek $k_i$, $\ell_i$ and $W_i$ (and also $m_{(i-1)/2}$ if $i$ is odd) such that $(1)-(8)$ hold for $i$.

By Lemma \ref{lem:bar_d_con} with $W_{i-1}$ as $X$ and $k_{i-1}$ as $d$, let $I \subseteq [-1, 1]$ be an interval such that $\sin W_{i-1} \in I$ and for all $Y \in \N$,
\[ \sin Y \in I \implies Y|_{k_{i-1}} = W_{i-1} . \]
Pick $T \in \N$ such that $\sin T$ lies in the interior of $I$. Then
\begin{itemize}
    \item $\sin W_{i-1} < \sin T$,
    \item $T|_{k_{i-1}} = W_{i-1}$, and
    \item for all $Z \in \N$, \[ \sin W_{i-1} < \sin Z < \sin T \implies Z|_{k_{i-1}} = W_{i-1} . \]
\end{itemize}
By Lemma \ref{lem:bar_d_best} with $d = k_{i-1}$, $s = c_i$, and $\overline{X} = (W_{i-1}, T)$, let $\ell_i,e_2,W_i \in \N$ be such that
\begin{itemize}
    \item $\sin W_{i-1} < \sin W_i < \sin T$,
    \item $k_{i-1} < \ell_i < e_2$,
    \item $W_i \leq e_2$,
    \item ${\bf Next}(\ell_{i-1},\ell_i)$, and
    \item for all $Z \in \N$,
        \begin{equation}\label{eq:Best_impl}
            Z|_{e_2} = W_i \implies {\bf Best}(k_{i-1}, \ell_i, W_{i-1}, Z, c_i) . \tag{$\star$}
        \end{equation}
\end{itemize}

We again invoke Lemma \ref{lem:next}. If $i > 1$ is odd, pick $m_{(i-1)/2} \geq k_{i-1}$ such that ${\bf Next}(m_{(i-1)/2 - 1}, m_{(i-1)/2})$ holds.
Finally, pick $k_i > \max\{W_i, \ell_i, m_{\lfloor(i-1)/2\rfloor}, e_2\}$ such that ${\bf Next}(k_{i-1}, k_i)$ holds.
We now check the necessary conditions.

For (1), notice $W_i|_{k_i} = W_i$ since $k_i > W_i$.
Observe that by choice of $T$ and since $\sin W_{i-1} < \sin W_i < \sin T$, we have $W_i|_{k_{i-1}} = W_{i-1}$.
Let $j \in \{1,\dots,i-1\}$.
Recall that $W_{i-1}|_{k_j} = W_j$. Then since $k_j \leq k_{i-1}$, we have $W_i|_{k_j} = \big(W_i|_{k_{i-1}}\big)|_{k_j} = W_{i-1}|_{k_j} = W_j$ as desired.

For (2)-(3), notice $k_i > \ell_i \geq e_1 > k_{i-1} > \max\{c_1,\dots,c_{2n}\}$.

For (4), notice that $k_i \geq W_i$ holds by construction. So let $Z \in \N$ be such that $Z|_{k_i} = W_i$. Since $k_i > e_2 \geq W_i$, we have $Z|_{e_2} = W_i$. Then by \eqref{eq:Best_impl}, ${\bf Best}(k_{i-1}, \ell_i, W_{i-1}, Z, c_i)$ holds.

Notice (5)-(8) hold by construction.

Induction is complete, and we have thus constructed $(k_i)_{i=0}^{2n}$, $(\ell_i)_{i=0}^{2n}$, $(m_j)_{j=0}^n$, and $(W_i)_{i=0}^{2n}$ satisfying (1)-(8) for each $i = 0,\dots,2n$.
\end{proof}

\begin{thm}\label{thm:member}
Let $S \subseteq \N^2$ be finite. Then there are $X_1, X_2, X_3, X_4 \in \N$ such that for all $s,t \in \N$,
\[ (s,t) \in S \iff \Member(X_1,X_2,X_3,X_4,s,t).\]
\end{thm}
\begin{proof}
Let $(k_i)_{i=0}^{2n}$, $(\ell_i)_{i=0}^{2n}$, $(m_j)_{j=0}^n$, and $(W_i)_{i=0}^{2n}$ be the sequences from Lemma \ref{lem:member_seq} based on $S$. We refer to the properties of these sequences according to their enumeration in the statement of Lemma \ref{lem:member_seq}. By conditions (5), (6), and (8) the sequences $(k_i)_{i=0}^{2n}$, $(\ell_i)_{i=0}^{2n}$, and $(m_j)_{j=0}^n$ are convergent in the sense of best approximates; namely, the terms of each sequence are consecutive best approximates for the final term. Define $Z_1,Z_2,Z_3,Z_4 \in \N$ as follows:
\[ Z_1 := k_{2n} \hspace{1cm} Z_2 := \ell_{2n} \hspace{1cm} Z_3 := m_n \hspace{1cm}  Z_4 := W_{2n} \]
We will show that for all $s,t \in \N$, \[ (s,t) \in S \iff \Member(Z_1, Z_2, Z_3, Z_4, s, t).\]

$(\!\implies\!)$ Let $(s,t)\in S$. Let $i \in \{1,\dots,2n\}$ be such that $(s,t) = (c_i,c_{i+1})$; so $i$ is odd. We seek that $\Admissible(k_{i-1},k_i,k_{i+1},k_{i-1}, \ell_i, \ell_{i+1}, Z_1, Z_2, Z_3, Z_4, c_i, c_{i+1})$ holds. 
By property (3), we have $k_{i-1}<\ell_i<k_i<\ell_{i+1}<k_{i+1}$. By property (1) and since $Z_4 = W_{2n}$, we have:
\[ Z_4|_{k_{i-1}} = W_{i-1}, \hspace{1cm} Z_4|_{k_i} = W_i, \hspace{1cm} Z_4|_{k_{i+1}} = W_{i+1} . \]

Then by property (4), we have
\[ {\bf Best}(k_{i-1}, \ell_i, Z_4|_{k_{i-1}}, Z_4, c_i) \text{ and } {\bf Best}(k_i, \ell_{i+1}, Z_4|_{k_i}, Z_4, c_{i+1}) \]
so that
\[ \Admissible(k_{i-1},k_i,k_{i+1},k_{i-1}, \ell_i, \ell_{i+1}, Z_1, Z_2, Z_3, Z_4, c_i, c_{i+1}) \]
holds and thus also $\Member(Z_1, Z_2, Z_3, Z_4, s, t)$.

$(\!\impliedby\!)$ Let $(s,t) \in \N^2$ be such that $\Member(Z_1, Z_2, Z_3, Z_4, s, t)$ holds. Let $d_1, d_2, d_3, d_4, e_1, e_2 \in \N$ be such that $\Admissible(d_1, d_2, d_3, d_4, e_1, e_2, Z_1, Z_2, Z_3, Z_4,\linebreak s, t)$ holds.
Then $d_1 < d_2 < d_3$ with $d_1,d_2,d_3$ as consecutive best approximates of $Z_1$.
Set $k_{-1} = 0$ so that ${\bf Next}(k_{-1}, k_0)$ holds.
Then by construction of $Z_1$, property (i) of $\Admissible$, and Lemma \ref{lem:next_seq}, there is some index $i \in \{0,1,\dots,2n-1\}$ such that $d_1 = k_{i-1}$, $d_2 = k_i$, and $d_3 = k_{i+1}$.

Note that $d_4$ is a best approximate of $Z_3$. Set $m_{-1}=0$ so that ${\bf Next}(m_{-1},m_0)$ holds. Then by construction of $Z_3$ and Lemma \ref{lem:next_seq}, there is some index $j \in \{-1,0,1,\dots,n\}$ such that $d_4 = m_j$.
By property (ii) of $\Admissible$, we have \[ d_1 = k_{i-1} \leq d_4 = m_{j} < d_2 = k_i. \]
Then by property (7), we have $j = (i-1)/2$, so $i$ must be odd.

Similarly $e_1$ and $e_2$ are best approximates of $Z_2$, and $e_2$ is the better approximate.
By construction of $Z_2$ and Lemma \ref{lem:next_seq}, there are indices $j_1,j_2 \in \{0,1,\dots,2n\}$ with $j_1 < j_2$ such that $e_1 = \ell_{j_1}$ and $e_2 = \ell_{j_2}$.
By property (iii) of $\Admissible$, we have \[ d_1 = k_{i-1} \leq e_1 = \ell_{j_1} < d_2 = k_i \leq e_2 = \ell_{j_2} < d_3 = k_{i+1} . \]
Then by property (3), $j_1 = i$ and $j_2 = i+1$. So $e_1 = \ell_i$ and $e_2 = \ell_{i+1}$.

Now by property (iv) of $\Admissible$, we have ${\bf Best}(k_{i-1}, \ell_i, Z_4|_{k_{i-1}}, Z_4, s)$. Similarly by property (v) of $\Admissible$, we have ${\bf Best}(k_i, \ell_{i+1}, Z_4|_{k_i}, Z_4, t)$.
By property (4) and the uniqueness of the final argument of {\bf Best}, these yield that $s = c_i$ and $t = c_{i+1}$. Since $i$ is odd, $(s, t) = (c_i, c_{i+1}) \in S$ as desired.
\end{proof}

\begin{lem}\label{lem:member-def}
$\Admissible$ is definable by an $\exists^4\forall^{14}$ $\sin$-PA formula and \Member is definable by an $\exists^{10}\forall^{14}$ $\sin$-PA formula.
\end{lem}
\begin{proof}
Notice the following $\sin$-PA formula defining $\Admissible$:
\begin{align*}
 {\bf Admi}&{\bf ssible}(d_1,d_2,d_3,d_4, e_1,e_2, X_1,X_2,X_3,X_4, s,t) = \\
 &d_1 < d_2 < d_3 \wedge X_1|_{d_1} = d_1 \wedge X_1|_{d_2} = d_2 \wedge X_1|_{d_3} = d_3 \\
 &\wedge X_1|_{d_2 - 1} = d_1 \wedge X_1|_{d_3 - 1} = d_2 \\
 &\wedge d_1 \leq d_4 \leq d_2 \wedge X_3|_{d_4} = d_4 \\
 &\wedge d_1 \leq e_1 < d_2 \leq e_2 < d_3 \wedge X_2|_{e_1} = e_1 \wedge X_2|_{e_2} = e_2 \\
 &\wedge {\bf Best}(d_1, e_1, X_4|_{d_1}, X_4, s) \wedge {\bf Best}(d_2, e_2, X_4|_{d_2}, X_4, t),
\end{align*}

Similar to {\bf Better}, we introduce the $\sin$-PA formula {\bf BetterDiff} to capture quantifier-free part of {\bf Best} (see Lemma \ref{lem:best_def}); that is, ${\bf BetterDiff}(d,e,\overline{X},\overline{Y},\overline{Z})$ will hold if and only if the pair $\overline{Z}$ is exactly $\overline{Y}$ or $\overline{Y}$ comprises a better difference approximate of $\overline{X}$ than does $\overline{Z}$ (under sine and up to $d$ and $e$).
\begin{align*}
 &{\bf BetterDiff}(d,e,X_1,X_2,Y_1,Y_2,Z_1,Z_2) := Y_1 \leq d \wedge Y_2 \leq e \\
 &\hspace{0.2cm} \wedge \big[ Z_1 \leq d \wedge Z_2 \leq e \wedge (Z_1 \neq Y_1 \vee Z_2 \neq Y_2) \big] \rightarrow \\
 &\hspace{1cm} \big| \sin X_2 - \sin X_1 - |\sin Y_2 - \sin Y_1|\big| < \big| \sin X_2 - \sin X_1 - |\sin Z_2 - \sin Z_1|\big|.
\end{align*}

Using these formulas, we may express $\Admissible$ in prenex form:
\begin{align*}
 &\exists Y_1, Y_2, Y_3, Y_4 \in \N
 \ \forall Z_1, Z_2, Z_3, Z_4, Z_5, Z_6, Z_7, Z_8, Z_9, Z_{10}, Z_{11}, Z_{12}, Z_{13}, Z_{14} \in \N \\ 
 &\ d_1 \leq d_4 < d_2 < d_3 \wedge d_1 \leq e_1 < d_2 \leq e_2 < d_3 \\ 
 &\wedge {\bf Better}(d_1,X_1,d_1,Z_1)
 \wedge {\bf Better}(d_2,X_1,d_2,Z_2) 
 \wedge {\bf Better}(d_3,X_1,d_3,Z_3) \\ 
 &\wedge {\bf Better}(d_2-1,X_1,d_1,Z_4)
 \wedge {\bf Better}(d_3-1,X_1,d_2,Z_5) \\ 
 &\wedge {\bf Better}(d_4,X_3,d_4,Z_6)
 \wedge {\bf Better}(e_1,X_2,e_1,Z_7) 
 \wedge {\bf Better}(e_2,X_2,e_2,Z_8) \\ 
 &\wedge {\bf Better}(d_1,X_4,Y_1,Z_9)
 \wedge {\bf Better}(d_2,X_4,Y_3,Z_{12}) \\ 
 &\wedge {\bf BetterDiff}(d_1,e_1,Y_1,X_4,s,Z_{10},Z_{11}) \\ 
 &\wedge {\bf BetterDiff}(d_2,e_2,Y_3,X_4,t,Z_{13},Z_{14}) 
\end{align*}

As shown, $\Admissible$ is $\exists^4\forall^{14}$-definable, so \Member is $\exists^{10}\forall^{14}$-definable.
\end{proof}

\begin{proof}[Proof of Theorem A]
    We apply Fact \ref{fact:hpn} with \Member in place of $P$.
    Recall that $\Member$ can be defined by an $\exists^{10}\forall^{14}$ $\sin$-PA formula, so the set of $\exists^K\forall^K\exists^K\forall^K$ $\sin$-PA sentences, where $K = 3388$, is undecidable. 
\end{proof}

\section{Lower bound for decidability}

In this section, we prove Theorem B. This means that under Schanuel's conjecture, we give a decision procedure for the set of all existential $\sin$-PA sentences.
\newline

The decision procedure begins with effective reductions to a $\sin$-PA sentence whose matrix formula bears a single type of $\mathcal{L}_{\sin}$-literal: first we eliminate sine from equalities, then eliminate all equalities, and finally eliminate variables outside the scope of sines. The resulting matrix formula is equisatisfiable with the original and bears only inequalities between constants and sums of sines.
Sentences of this type are then decided by defining a set of real tuples which serve as proxies to solutions of the matrix formula over the integers.
The nonemptiness of this set is $\mathcal{L}_{\sin}$-definable, so we complete the decision according to the decidability of $\FO(\R, <, +, \sin)$ granted by Schanuel's conjecture.

\subsection{Tools.}
Fix variables $x = (x_1,\dots,x_n)$ and let $\cdot$ denote the dot product.
Notice that for every $\mathcal{L}_{\sin}$-term, there are $p_0 \in \Q^{n+1}$, $p_1,\dots,p_K \in \Q^{m+n+1}$, and $r_1,\dots,r_K \in \Q_{\neq 0}$ with $m,K \geq 0$ such that the term may be written in the form
\begin{equation}\label{eq:t_form}\tag{$\dagger$}
p_0 \cdot (x,1) + \sum\limits_{i=1}^K r_i \sin \Big( p_i \cdot \big(x, 1, t_1(x), \dots, t_m(x)\big) \Big) ,
\end{equation}
where $t_i(x)$ is an $\mathcal{L}_{\sin}$-term written as
\[ t_i(x) = \sin\Big( q_i\cdot \big(x,1, t_1(x), \dots, t_{i-1}(x)\big)  \Big) \]
with $q_i \in \Q^{n+i}$ for each $i=1,\dots,m$.

\begin{defn}
An $\mathcal{L}_{\sin}$-term is {\bf oscillatory} if when written as in \eqref{eq:t_form}, each entry of $p_{0}$ is $0$.
\end{defn}
Let $\mathcal{L} = \{+, 0, 1, (\lambda_q)_{q \in \Q} \}$. Note that an oscillatory $\mathcal{L}_{\sin}$-term is either written as a sum of sines, with no outermost summand being an $\mathcal{L}$-term, or is 0.

\begin{defn}\label{def:radius}
The {\bf radius} of an oscillatory $\mathcal{L}_{\sin}$-term $t(x)$ written as in \eqref{eq:t_form} is \[ R(t) := \sum_{i=1}^K |r_i| . \]
\end{defn}
\begin{fact}\label{fact:radius}
Let $t(x)$ be an oscillatory $\mathcal{L}_{\sin}$-term.
Then for all $z \in \Q^n$, $|t(z)| \leq R(t)$.
Equality holds if and only if $R(t) = 0$.
\end{fact}


\begin{defn}\label{def:literals}
Let $q = (q_1,\dots,q_{n+1}) \in \Q^{n+1}$ and let $t(x)$ be an oscillatory $\mathcal{L}_{\sin}$-term.
For convenience, we name the following types of $\mathcal{L}_{\sin}$-literal:
\begin{itemize}
    \item\label{item:ineq} $q\cdot (x,1) < t(x)$ is an {\bf $\mathcal{L}_{\sin}$-inequality}, 
        \item\label{item:osc-ineq} $q_{n+1} < t(x)$ is an {\bf oscillatory $\mathcal{L}_{\sin}$-inequality}, 
    \item\label{item:eq} $q\cdot (x,1) + t(x) = 0$ is an {\bf $\mathcal{L}_{\sin}$-equality}, 
        \item\label{item:lin-eq} $q\cdot (x,1) = 0$ is an {\bf $\mathcal{L}$-equality}, 
    \item\label{item:diseq} $q\cdot (x,1) + t(x) \neq 0$ is an {\bf $\mathcal{L}_{\sin}$-disequality}, and 
        \item \label{item:lin-diseq} $q \cdot (x,1) \neq 0$ is an {\bf $\mathcal{L}$-disequality}. 
\end{itemize}
Notice that any Boolean combination $\Phi$ of $\mathcal{L}_{\sin}$-literals is a $\sin$-PA formula.
\end{defn}

\begin{figure}
\begin{flushleft}
\textsc{Decide-existential-sin-PA$(\exists x \in \Z^n\ \Phi(x)$)} \\
\textsc{Input:} existential $\sin$-PA sentence $\exists x \in \Z^n\ \Phi(x)$, with $\Phi(x)$ a quantifier-free conjunction of positive $\mathcal{L}_{\sin}$-literals \\

\begin{enumerate}[1.\ ]
 
 \item \textsc{for each} $\mathcal{L}_{\sin}$-equality \textsc{in} $\Phi$:
 \item \ind Replace by combination of $\mathcal{L}$-equalities and -disequalities
 
 \item Adjust $\Phi$ to disjunctive normal form
 \item \textsc{for each} conjunctive clause \textsc{in} $\Phi$:
 \item \ind Replace by combination of $\mathcal{L}_{\sin}$-inequalities and divisibility predicates
 
 \item Adjust $\Phi$ to disjunctive normal form
 \item \textsc{for each} conjunctive clause \textsc{in} $\Phi$:
 \item \ind Replace by combination of oscillatory $\mathcal{L}_{\sin}$-inequalities and divisibility predicates
 
 \item Adjust $\Phi$ to disjunctive normal form
 \item \textsc{for each} conjunctive clause \textsc{in} $\Phi$:
 \item \ind Replace by combination of oscillatory $\mathcal{L}_{\sin}$-inequalities
 
 \item Obtain $\mathcal{L}_{\sin}$-sentence $\theta$ encoding nonemptiness of proxy solution set $M_\Phi$
 \item \textsc{return} whether $\left(\R, <, +, \sin\right) \models \theta$
 
 \end{enumerate}
 
\textsc{Output:} Boolean answer whether or not $\exists x \in \Z^n\ \Phi(x)$ holds
\end{flushleft}
\caption{Decision algorithm for existential $\sin$-PA sentences.} \label{fig:alg}
\end{figure}

Figure \ref{fig:alg} presents pseudocode for the decision procedure. Schanuel's conjecture enables two critical steps.
In the reduction of literals, it enables transforming $\mathcal{L}_{\sin}$-equalities into Boolean combinations of $\mathcal{L}$-equalities.
For the proxy solution set at the end, it renders the theory $\FO\left(\R, <, +, \sin\right)$, and thus the nonemptiness query, decidable.

\subsection{Reduction with Schanuel's conjecture.}\label{sec:reduction_SC}
The goal of this subsection is to prove Theorem \ref{thm:elim-Lsin-equality}, which will allow us to express linear-sine equalities as Boolean combinations of strictly linear equalities.

\begin{thm}\label{thm:elim-Lsin-equality}
Assume Schanuel's conjecture. 
Let $\varphi(x)$ be an $\mathcal{L}_{\sin}$-equality.
Then there is a positive Boolean combination $\Psi(x)$ of $\mathcal{L}$-equalities and -disequalities such that for $z \in \Q^n$,
\[ \varphi(z) \text{ holds if and only if } \Psi(z) \text{ holds}. \]
Moreover, $\Psi$ can be computed from $\varphi$.
\end{thm}

\noindent Throughout the subsection, fix a particular $\mathcal{L}_{\sin}$-equality $\varphi(x)$ written as
\[ p_0 \cdot (x,1) + t(x) = 0 \]
where $p_0 = (p_{0,1},\dots,p_{0,n+1}) \in \Q^{n+1}$ and $t$ is an oscillatory $\mathcal{L}_{\sin}$-term.
Fix $p_1,\dots,p_K \in \Q^{m+n+1}$, $r_1,\dots,r_K \in \Q_{\neq 0}$, minimal $K \geq 0$, minimal $m\geq 0$, and $\mathcal{L}_{\sin}$-terms $t_i(x)$ with $q_i \in \Q^{n+i}$ for $i=1,\dots,m$ as in \eqref{eq:t_form} applied to the $\mathcal{L}_{\sin}$-term $p_0 \cdot (x,1) + t(x)$.
For $i=1,\dots,K$ denote the sine arguments
\[ P_i(x) := p_i \cdot \big( x, 1, t_1(x),\dots,t_m(x)\big) . \]

\begin{lem}\label{lem:subterms-algdep}
Let $z \in \Q^n$ be such that $\varphi(z)$ holds. Then $\sin 1, \sin t_1(z), \dots, \sin t_m(z)$ are algebraically dependent.
\end{lem}
\begin{proof}
    Let $X = (X_1,\dots,X_{m+n+1})$ be complex variables and define the function \[ f(X) := \sum\limits_{k=1}^{n+1} p_{0,k} X_k + \sum\limits_{j=1}^K \frac{r_j}{2\mathbbm{i}} \Bigg( \prod\limits_{k=1}^{m+n+1} X_k^{p_{j,k}} \Bigg) + \sum\limits_{j=1}^K \frac{-r_j}{2\mathbbm{i}} \Bigg( \prod\limits_{k=1}^{m+n+1} X_k^{-p_{j,k}} \Bigg) . \]
    Since $K$ is minimal, for every distinct $j,j' \in \{1,\dots,K\}$, the vector $p_j$ has some nonzero coordinate and there is some coordinate $k \in \{1,\dots,m+n+1\}$ such that $|p_{j,k}| \neq |p_{j',k}|$.
    Thus no monomials (allowing for exponents from $\Q$) in $f$ cancel out, so $f$ is nonzero.
    Consider \begin{align*}
        t(x) 
        &= \sum\limits_{j=1}^K r_j\sin\big( P_j(x) \big) \\
        &= \sum\limits_{j=1}^K \frac{r_j}{2\mathbbm{i}}\Big( \exp\!\big(\mathbbm{i} P_j(x)\big) - \exp\!\big(-\mathbbm{i} P_j(x)\big) \Big) \\
        &= \sum\limits_{j=1}^K \frac{r_j}{2\mathbbm{i}} \Bigg( \bigg[ \prod\limits_{k=1}^n \exp(\mathbbm{i}x_k)^{p_{j,k}} \bigg] \cdot \exp(\mathbbm{i})^{p_{j,n+1}} \cdot \bigg[ \prod\limits_{k=1}^m \exp(\mathbbm{i}t_k(x))^{p_{j,k+n+1}} \bigg] \\
        &\hspace{1.5cm} - \bigg[ \prod\limits_{k=1}^n \exp(\mathbbm{i}x_k)^{-p_{j,k}} \bigg] \cdot \exp(\mathbbm{i})^{-p_{j,n+1}} \cdot \bigg[ \prod\limits_{k=1}^m \exp(\mathbbm{i}t_k(x))^{-p_{j,k+n+1}} \bigg] \Bigg) .
    \end{align*}
    Define the tuple
    \[ z_{\exp} := \Big( \exp(\mathbbm{i}z_1),\dots,\exp(\mathbbm{i}z_n), \exp(\mathbbm{i}), \exp\!\big(\mathbbm{i}t_1(z)\big), \dots, \exp\!\big(\mathbbm{i}t_m(z)\big) \Big) \]
    and notice $f(z_{\exp}) = p_0\cdot(z,1) + t(z)$. Since $\varphi(z)$ implies $p_0\cdot(z,1) + t(z) = 0$, we have that $z_{\exp}$ is a root of $f$.
    Since each $p_{j,k}$ is rational, we may manipulate $f$ into a nonzero complex polynomial $g$ such that $g(z_{\exp}) = 0$.
    Hence $\exp(\mathbbm{i}z_1),\dots,\exp(\mathbbm{i}z_n), \exp(\mathbbm{i}), \exp\!\big(\mathbbm{i}t_1(z)\big), \dots, \exp\!\big(\mathbbm{i}t_m(z)\big)$ are algebraically dependent.
    By the proof of Fact \ref{fact:sinesc}, we have that $\exp ( \mathbbm{i}\alpha)$ and $\sin \alpha$ are algebraically dependent for every $\alpha \in \C$.
    So $\sin z, \sin 1, \sin t_1(z),\dots, \sin t_m(z)$ are algebraically dependent.
    Since $z$ is rational, the lemma follows.
\end{proof}

\begin{lem}\label{lem:nestedsine_LW}
Assume Schanuel's conjecture. Let $z \in \Q^n$ be such that $\varphi(z)$ holds.
If $K > 0$, then at least one of the following holds: 
\begin{enumerate}[(i)]
    \item There is $i\in\{1,\dots,K\}$ 
    such that \[ p_i \cdot \big(z,1,t_1(z),\dots,t_m(z)\big) = 0 . \]
    \item There are distinct $i,j \in \{1,\dots,K\}$ such that 
    \[ \big|p_i \cdot \big(z,1,t_1(z),\dots,t_m(z)\big)\big| = \big|p_j \cdot \big(z,1,t_1(z),\dots,t_m(z)\big)\big| . \]
\end{enumerate}
\end{lem}
\begin{proof}
    We proceed by induction on $m$. Let $m=0$. 
    Assume $(ii)$ does not hold, so the arguments of any two sines from $t(z)$ bear different absolute values.
    Since $z$ is a rational tuple, Fact \ref{fact:LW-sine} yields that some sine argument is zero; that is, $(i)$ holds.
    
    \noindent For the induction step, let $m > 0$.
    We obtain: \begin{align*}
        \ldim_{\Q}\big(1, & \ t_1(z), \dots, t_m(z) \big) \\
        &= \ldim_{\Q}\big(z, 1, t_1(z), \dots, t_m(z) \big) \\
        &\leq \td_{\Q}\big(z, \sin z, 1, \sin 1, t_1(z), \sin t_1(z), \dots, t_m(z), \sin t_m(z) \big) \\
        &= \td_{\Q}\big(\sin 1, \sin t_1(z), \dots, \sin t_m(z) \big) \\
        &\leq m .
    \end{align*}
    Indeed, the first inequality follows from Fact \ref{fact:sinesc} and thus from Schanuel's conjecture.
    The last equality follows from Fact \ref{fact:td_zero} applied to each $t_i(z)$ and $\sin z$ in the form of \eqref{eq:t_form}.
    The final inequality follows from Lemma \ref{lem:subterms-algdep}.
    Comparing extrema, we have $\ldim_{\Q}(1, t_1(z), \dots, t_m(z) ) \leq m$,
    so let $k\in \{1,\dots,m\}$ be minimal such that  
    \[ \ldim_{\Q}(1, t_1(z), \dots, t_k(z) ) \leq k . \]
    Then there is $c = (c_0,c_1,\dots,c_k) \in \Q^{k+1}$ with $c_k \neq 0$ such that
    \[ c_0 + \sum\limits_{i=1}^k c_i t_i(z)=0 \]
    and accordingly some $v = (v_0,v_1,\dots,v_{k-1}) \in \Q^{k}$ such that
    \begin{equation*}\label{eq:lemLWGtk}\tag{$\ast$}
    t_k(z) = v \cdot \big(1,t_1(z),\dots,t_{k-1}(z)\big) .
    \end{equation*}
    
    From each $p_i$, we now construct a vector $u_i = (u_{i,1},\dots,u_{i,m+n}) \in \Q^{m+n}$ which essentially replaces the contribution from $t_k(z)$ according to \eqref{eq:lemLWGtk}. Define
    \[
    u_{i,j} := \begin{cases}
        p_{i,j} & 1 \leq j \leq n\\
        p_{i,j} + p_{i,k+n+1}v_{j-n-1} & n < j < k+n+1 \\
        p_{i,j+1} & k+n+1 \leq j \leq m+n .
    \end{cases}
    \]
    Hence for each $i=1,\dots,K$ we obtain
    \begin{equation*}\label{eq:lemLWG}\tag{$\ast\ast$}
    p_i\cdot \big(z, 1, t_1(z), \dots, t_{m}(z)\big) = u_i \cdot \big(z, 1, t_1(z), \dots, \widehat{t_{k}(z)},\dots, t_{m}(z)\big)
    \end{equation*}
    where $\widehat{t_k(z)}$ denotes omission of $t_k(z)$ from the tuple.
    Notice that $T(z) = p_0\cdot(z,1) + t(z) = 0$, where
    \[
    T(x) := p_0 \cdot (x, 1) + \sum\limits_{i=1}^K r_i \sin \big( u_i \cdot (x, 1, t_1(x), \dots, \widehat{t_{k}(z)}, \dots, t_{m}(x)) \big).
    \]
    By the induction hypothesis applied to $T(z)$, one of the following holds:
    \begin{enumerate}[$(i^{\ast})$]
        \item There is $i\in\{1,\dots,K\}$ 
        such that \[ u_i \cdot \big(z,1,t_1(z),\dots,\widehat{t_{k}(z)}, \dots, t_{m}(z)\big) = 0 . \]
        \item There are distinct $i,j \in \{1,\dots,K\}$ such that 
        \[ \big| u_i \cdot \big(z,1,t_1(z),\dots,\widehat{t_{k}(z)}, \dots, t_{m}(z)\big)\big| = \big| u_j \cdot \big(z,1,t_1(z),\dots,\widehat{t_{k}(z)}, \dots, t_{m}(z)\big)\big| . \]
    \end{enumerate}
    By \eqref{eq:lemLWG}, we have that $(i^{\ast})$ implies $(i)$ and similarly $(ii^{\ast})$ implies $(ii)$.
\end{proof}

\begin{lem}\label{lem:r_zero}
Assume Schanuel's conjecture. Let $z \in \Q^n$. Then $\varphi(z)$ holds if and only if $p_0 \cdot (z, 1) = 0$ and for each $i=1,\dots,K$,
\begin{equation*}\label{eq:r_zero}\tag{$\star$}
\sum\limits_{j \in J_i^+(z)} r_j - \sum\limits_{j \in J_i^-(z)} r_j = 0
\end{equation*}
where $J_i^\pm(z)$ is the set of $j \in \{1,\dots,K\}$ such that $P_i(z) = \pm P_j(z)$.
\end{lem}
\begin{proof}
    We proceed by induction on $K$. If $K=0$, then $t(x) = 0$. Thus $\varphi(z)$ holds if and only if $p_0 \cdot (z,1) = 0$.
    
    \noindent For the induction step, let $K > 0.$
    Suppose $\varphi(z)$ holds, so $p_0 \cdot (z,1) + t(z) = 0$.
    Assume by Lemma \ref{lem:nestedsine_LW} that $(i)$ holds for $k$, so $P_k(z) = 0$.
    Then
    \[ \sum\limits_{\substack{i=1\\i\neq k}}^{K} r_i\sin P_i(z) = t(z) . \]   
    By the inductive hypothesis, $p_0 \cdot (z,1) + t(z) = 0$ gives that $p_0\cdot (z,1) = 0$ and for each $i=1,\dots,K$ with $i\neq k$:
    \[ \sum\limits_{\substack{j\in J_i^+(z)\\j\neq k}} r_j - \sum\limits_{\substack{j\in J_i^-(z)\\j\neq k}} r_j = 0 . \] 
    Since $P_K(z) = 0$, we have that $k \in J_i^+(z)$ if and only if $k \in J_i^-(z)$. So by allowing for $j=k$ in the sums of the above equation, \eqref{eq:r_zero} holds for each $i\neq k$.
    Further, $J_k^+(z) = J_k^-(z)$ so \eqref{eq:r_zero} also holds for $i = k$.
    
    Now assume by Lemma \ref{lem:nestedsine_LW} that $(ii)$ holds for distinct $k,k'$, so $|P_k(z)| = |P_{k'}(z)|$. Set \[
    r'_{k} := r_{k} + \begin{cases}
        r_{k'} & P_k(z) = P_{k'}(z) \\
        -r_{k'} & P_k(z) = -P_{k'}(z) .
    \end{cases}
    \]
    Then the first case in which $(i)$ from Lemma \ref{lem:nestedsine_LW} holds now applies to \[ r'_k\sin P_k(z) + \sin 0 +  \sum\limits_{\substack{i=1\\i\neq k,k'}}^K r_i \sin P_i(z) = t(z) . \]
    That is, the case with $(ii)$ reduces to the case with $(i)$ above.
    Hence $p_0\cdot(z,1) + t(z) = 0$ implies $p_0\cdot(z,1) = 0$ and that \eqref{eq:r_zero} holds for each $i=1,\dots,K$.
    
    For the reverse implication, suppose $p_0\cdot(z,1) = 0$ and that \eqref{eq:r_zero} holds for each $i=1,\dots,K$.
    Since sine is an odd function, gathering summands from $t(z)$ according to the index sets from \eqref{eq:r_zero} makes clear that the latter supposition yields $t(z) = 0$. Thus $p_0\cdot(z,1) + t(z) = 0$ and $\varphi(z)$ holds.
\end{proof}

\begin{defn}\label{def:congrel}
    Fix the unary system equipped with negation \[ Z := \big\{0, P_1(x), -P_1(x), \dots, P_K(x), -P_K(x) \big\} . \]
    Fix $z \in \Q^n$. Then let $\sim_z$ be the congruence relation over $Z$ such that for each $i=1,\dots,K$ we have
    \[ P_i(x) \sim_{z} \pm P_j(x) \text{ if and only if } P_i(z) = \pm P_j(z) . \]
\end{defn}

\begin{fact}\label{fact:congrel}
Let $z,z' \in \Q^n$ be such that $\sim_z\,=\, \sim_{z'}$.
Then for each $i=1,\dots,K$ \[ z \text{ satisfies } \eqref{eq:r_zero} \text{ if and only if } z' \text{ satisfies } \eqref{eq:r_zero} . \] 
\end{fact}
\begin{proof}
    Critically, whether $z$ satisfies \eqref{eq:r_zero} for $i$ depends on the index sets $J_i^+(z)$ and $J_i^-(z)$ induced by $z$, not on the particular values of $z$ or any $P_j(z)$. Indeed, these index sets correspond to the classes of $\sim_z$.
    
    Fix $i \in \{1,\dots,K\}$ and observe that $i \in J_i^-(z)$ if and only if $P_i(z) = 0$ if and only if $J_i^+(z) = J_i^-(z)$.
    Further, for each $j = 1,\dots,K$:
    \begin{align*}
        &j \in J_i^+(z) \iff P_i(x) \sim_z P_j(x) \iff -P_i(x) \sim_z -P_j(x) , \\
        &j \in J_i^-(z) \iff P_i(x) \sim_z -P_j(x) \iff -P_i(x) \sim_z P_j(x) .
    \end{align*}
    Since $\sim_z\,=\,\sim_{z'}$, we have
    \[ j \in J_i^+(z) \iff P_i(x) \sim_z P_j(x)  \iff P_i(x) \sim_{z'} P_j(x) \iff j \in J_i^+(z') . \]
    Thus $J_i^+(z) = J_i^+(z')$. Similarly, $J_i^-(z) = J_i^-(z')$.
\end{proof}

\begin{defn}
In light of Fact \ref{fact:congrel}, we say that a congruence relation $\sim$ on $Z$ {\bf satisfies} \eqref{eq:r_zero} when for every $z \in \Q^n$, if $\sim_z\,=\,\sim$, then $z$ satisfies \eqref{eq:r_zero} for each $i=1,\dots,K$.
\end{defn}

\begin{defn}
Set $S_0(x) := \Q(x)$ and for $d \in \N$, define \[ S_{d+1}(x) := \Big\{ \sin \alpha : \alpha \in \Q(S_d(x))_{\neq 0} + \Q\Big(\bigcup\limits_{j<d} S_j(x)\Big) \Big\} . \]
The {\bf sine depth} of an $\mathcal{L}_{\sin}$-term over $x$ is the least index $d$ such that the term belongs to $\Q(\bigcup_{j \leq d} S_j(x))$.
\end{defn}

\begin{proof}[Proof of Theorem \ref{thm:elim-Lsin-equality}]
By Lemma \ref{lem:r_zero}, rational solutions to $\varphi(x)$ are precisely those that satisfy both $p_0\cdot (x,1) = 0$ and \eqref{eq:r_zero} for every $i=1,\dots,K$.
The constraint $p_0\cdot (x,1) = 0$ is already an $\mathcal{L}$-equality, so it remains to capture satisfaction of $\eqref{eq:r_zero}$ for every $i=1,\dots,K$ by an $\mathcal{L}$-formula; the conjunction of these will compose $\Psi$.

Let $d$ be the sine depth of $t(x)$; we proceed by induction on $d$.
If $d = 0$, then $K=0$.
Hence $\Psi(x)$ solely comprising $p_0\cdot(x,1) = 0$ suffices.

\noindent For the induction step, let $d > 0$.
By Fact \ref{fact:congrel}, any $z \in \Q^n$ satisfies \eqref{eq:r_zero} for every $i=1,\dots,K$ if and only if $\sim_z$ satisfies \eqref{eq:r_zero}.
It now suffices to encode each satisfactory congruence relation as an $\mathcal{L}$-formula.

Fix a congruence relation $\sim$ over $Z$ with classes $C_0,C_1,C_{-1},\dots,C_k,C_{-k}$ indexed so that $\alpha \in C_i$ if and only if $-\alpha \in C_{-i}$.
Since $\sim$ respects negation, we only need to encode $C_0,C_1,\dots,C_k$; pick respective class representatives $c_0(x),c_1(x),\dots,c_k(x)$. Observe that the $\mathcal{L}_{\sin}$-formula
\[ \theta_\sim(x) := \bigwedge\limits_{i=0}^k \Big( \bigwedge\limits_{j=0}^{i-1} \neg(c_i(x) - c_j(x) = 0) \wedge \bigwedge\limits_{c \in C_i} c_i(x) - c(x) = 0 \Big) \]
exactly captures the structure of $\sim$ by distinguishing class representatives while identifying members within each class. That is for $z \in \Q^n$, we have that $\theta_\sim(z)$ holds if and only if $\sim_z\,=\,\sim$.
Notice that each $\mathcal{L}_{\sin}$-term in $\theta_\sim(x)$ bears sine depth strictly less than $d$. So the inductive hypothesis yields an $\mathcal{L}$-formula $\theta'_\sim(x)$ which is equivalent to $\theta_\sim(x)$; we may apply De Morgan's law to maintain negation normal form.

We now construct $\Psi$. Consider the collection of all congruence relations over $Z$.
Since there are finitely many such, we may enumerate them and, since each coefficient $r_i$ is rational, decide whether each satisfies \eqref{eq:r_zero}.
Let $\sim_1,\dots,\sim_N$ be the relations that satisfy \eqref{eq:r_zero}, yielding $\mathcal{L}$-formulas $\theta'_{\sim_1}(x),\dots,\theta'_{\sim_N}(x)$ as described above. Set
\[ \Psi(x) := \big( p_0 \cdot (x,1) = 0 \big) \wedge \bigvee\limits_{i=1}^N \theta'_{\sim_i}(x) . \]

Let $z \in \Q^n$.
By Lemma \ref{lem:r_zero} and Fact \ref{fact:congrel}, we have that $\varphi(z)$ holds if and only if $p_0\cdot (z,1) = 0$ and $\sim_z$ satisfies \eqref{eq:r_zero}. By construction of $\Psi$, this occurs if and only if $\sim_z\,=\,\sim_i$ and $\theta'_{\sim_i}(z)$ for some $i \in \{1,\dots,N\}$.
Note that every quantifier-free $\mathcal{L}$-formula is a Boolean combination of $\mathcal{L}$-equalities; every such formula in negation normal form is a positive Boolean combination of $\mathcal{L}$-equalities and -disequalities.
\end{proof}

The correctness of the proof to Theorem \ref{thm:elim-Lsin-equality} relies on Schanuel's conjecture, but the construction of $\Psi$ does not.
In the induction step, there are at most $D_2(K)$ many congruence relations\footnote{$D_2(K)$ is the Dowling number with $m=2$ from \cite{Benoumhani}; $D_2(K) \in 2^{2^{O(K)}}$.}. 
Checking whether a given congruence relation satisfies \eqref{eq:r_zero} amounts to deciding whether for each pair of classes $C_i$ and $C_{-i}$, the corresponding sums of $r_j$ coefficients are equal. 
Each $\theta_\sim(x)$ contains $O(K^2)$ atoms.
Let $M$ be the number of appearances of sine in $\varphi$; so $M \geq \max (K,m,d)$.
By the recursive construction in the proof of Theorem \ref{thm:elim-Lsin-equality}, $\Psi$ contains $2^{2^{O(M)}}$ atoms. 

\begin{corollary}\label{cor:qf-sinPA}
Assume Schanuel's conjecture.
Then the set of quantifier-free $\sin$-PA sentences is decidable.
\end{corollary}
\begin{proof}
    It suffices to exhibit respective decision procedures for variable-free $\mathcal{L}_{\sin}$-equalities and -inequalities.
    A procedure for the former follows immediately from Theorem $\ref{thm:elim-Lsin-equality}$ and the decidability of Presburger arithmetic \cite{Presburger}.
    
    For the latter, let $q_{n+1} \in \Q$ and $t$ be a variable-free oscillatory $\mathcal{L}_{\sin}$-term; we will decide whether $q_{n+1} < t$.
    We first decide whether the variable-free $\mathcal{L}_{\sin}$-equality $q_{n+1} - t = 0$ holds, as above.
    If so, then the desired inequality does not hold.
    Otherwise, we invoke the Taylor series expansion to approximate $t$ by a polynomial arbitrarily well.
    Since $t \neq q_{n+1}$ in this case, we will eventually bound the approximation error away from $q_{n+1}$ to conclude whether or not $q_{n+1} < t$.
\end{proof}


The following corollary distinguishes $\mathcal{L}_{\sin}$-terms according to their sine depths.
\begin{corollary}
Assume Schanuel's conjecture. Then for $d > 0$, \[ \Q(S_d) \cap \Q \Big(\bigcup_{j<d} S_j \Big) = \{ 0 \} . \]
\end{corollary}
\begin{proof}
We proceed by induction on $d$. Fact \ref{fact:LW-sine} gives the result for $d=1$.

\noindent For the induction step, let $d > 1$.
Let $z \in \Q(S_d) \cap \Q \big(\bigcup_{j<d} S_j \big)$ and suppose $z$ is nonzero.
We may write $z$ as respective sums that we then equate:
\[ r_0 + \sum_{i=1}^m r_i \sin \alpha_i = r'_0 + \sum_{i=1}^{m'} r'_i \sin \beta_i , \]
with each $r_i,r_i' \in \Q$ nonzero (except possibly for $r_0,r_0'$), each $\sin \alpha_i \in S_d$, and each $\sin \beta_i \in S_{k_i}$ for $k_i < d$.
Assume $\alpha_i \neq 0$ and $|\alpha_i| \neq |\alpha_j|$ for each distinct $i,j$ (and analogously for $\beta_i,\beta_j$) since we could otherwise combine the corresponding summands into a single term.
Notice that each $\alpha_i \in S_{d-1}$ and each $\beta_i \in S_{k_i - 1}$ for some $k_i < d$.
By the inductive hypothesis, we have that $|\alpha_i| \neq |\beta_j|$ for every $i,j$.

Rearranging and relabeling the equality above, we obtain
\[ q_0 + \sum_{i=1}^K q_i \sin \gamma_i = 0 , \]
with $q_0 \in \Q$, $q_1,\dots,q_M \in \Q_{\neq 0}$, each $\gamma_i$ a nonzero $\mathcal{L}_{\sin}$-term of sine depth strictly less than $d$.
Notice that $K > 0$, each $\gamma_i \neq 0$, and $|\gamma_i| \neq |\gamma_j|$ for distinct $i,j$.
By Lemma \ref{lem:nestedsine_LW}, we obtain one of the following: \begin{enumerate}[$(i)$]
    \item There is $i \in \{ 1,\dots,K \}$ such that $\gamma_i = 0$.
    \item There are distinct $i,j \in \{ 1,\dots,K \}$ such that $|\gamma_i| = |\gamma_j|$.
\end{enumerate}
Each case contradicts the properties of $\gamma_i$ noted above, and we conclude that the only number in both $\Q(S_d)$ and $\Q \big(\bigcup_{j<d} S_j \big)$ is zero.
\end{proof}

\subsection{Reductions without Schanuel's conjecture.}
The goal of this subsection is to prove Theorems \ref{thm:elim-equality}, \ref{thm:elim-lin-from-inequality}, and \ref{thm:elim-divisibility}.
These form a pipeline which processes the types of $\mathcal{L}_{\sin}$-literal we must consider in the decision procedure to just oscillatory $\mathcal{L}_{\sin}$-inequalities. In contrast to those in Section \ref{sec:reduction_SC}, the results here do not assume Schanuel's conjecture.

\begin{defn}
Set the signature \[ \mathcal{L}_{\sin,\Z,D} := \mathcal{L}_{\sin} \cup \{ \Z \} \cup \{ D_k : k \in \Z_{\geq 2} \} , \] where each $D_k$ is a unary predicate symbol.
We name types of $\mathcal{L}_{\sin,\Z,D}$-literal, which include those of Definition \ref{def:literals} as well as, for $k \in \Z_{\geq 2}$ and $p = (p_1,\dots,p_{n+1}) \in \Z^{n+1}$:
\begin{itemize}
    \item\label{item:divis} $D_k\big(p\cdot (x,1)\big)$ is a {\bf divisibility predicate}. 
\end{itemize}
An $\mathcal{L}_{\sin,\Z,D}$-sentence containing a divisibility predicate is not a $\sin$-PA sentence, so we must consider its truth over an expanded structure.
Let $\sinPA_{\bf D}$ be the extension of $\sinPA$ defined analogously to $\sinPA$ from Section \ref{sec:formalize-sinPA} with the $\mathcal{L}_{\sin,\Z,D}$-structure $\left(\R, <, +, \sin, \Z, 2\Z, 3\Z, \dots \right)$, where $k\Z$ interprets $D_k$ for each $k \in \Z_{\geq 2}$, replacing $\left(\R, <, +, \sin, \Z\right)$ and divisibility predicates permitted in the matrix formulas.
\end{defn}

\begin{thm}\label{thm:elim-equality}
Let $\Phi(x)$ be a conjunction of $\mathcal{L}_{\sin}$-inequalities, $\mathcal{L}$-equalities and -disequalities, and divisibility predicates.
Then there is a positive Boolean combination $\Psi(x)$
of $\mathcal{L}_{\sin}$-inequalities and divisibility predicates such that
\[ \exists x \in \Z^n \ \Phi(x) \in \sinPA_{\bf D} \text{ if and only if } \exists x \in \Z^n\ \Psi(x) \in \sinPA_{\bf D} . \]
Moreover, $\Psi$ can be computed from $\Phi$.
\end{thm}
\begin{proof}
    We partition the atoms of $\Phi$ into conjunctions $\Phi^{<,D}$, $\Phi^=$, and $\Phi^{\neq}$ of $\mathcal{L}_{\sin}$-inequalities and divisibility predicates, $\mathcal{L}$-equalities, and $\mathcal{L}$-disequalities respectively. That is,
    \[ \Phi =: \Phi^{<,D} \wedge \Phi^= \wedge \Phi^{\neq} . \]
    
    Let $\chi_1,\dots,\chi_M$ enumerate the $\mathcal{L}$-disequalities composing $\Phi^{\neq}$. Fix $i \in \{1,\dots,M\}$ and suppose $\chi_i$ has form $q \cdot (x,1) \neq 0$ for $q = (q_1,\dots,q_{n+1}) \in \Q^{n+1}$. Let
    \[ \psi_i(x) := \big( q \cdot (x,1) < 0 \vee -q \cdot (x,1) < 0 \big) \]
    where $-q := (-q_1,\dots,-q_{n+1})$. By trichotomy, we have for $z \in \Z^n$ that $\chi_i(z)$ holds if and only if $\psi_i(z)$ holds.
    Set $\Psi^{\neq} := \bigwedge\limits_{i=1}^M \psi_i$ and $\Psi^{<,D} := \Phi^{<,D}$.
    
    Let $\varphi_1,\dots,\varphi_N$ enumerate the $\mathcal{L}$-equalities composing $\Phi^=$.
    We will construct $\mathcal{L}$-equalities $\varphi_{i,i+1}, \dots, \varphi_{i,N}$ and a conjunction $\Psi_i$ of $\mathcal{L}_{\sin}$-inequalities and divisibility predicates such that
    \[ \exists x \in \Z^n\ \Psi_i(x) \wedge \bigwedge\limits_{j=i+1}^N \varphi_{i,j}(x) \in \sinPA_{\bf D} \text{ if and only if } \exists x \in \Z^n\ \Phi(x) \in \sinPA_{\bf D} \]
    for each $i = 0,\dots,N$. We proceed by induction on $i$.
    
    Let $i = 0$. Set $\Psi_0 := \Psi^{<,D} \wedge \Psi^{\neq}$ if either $\Phi^{<,D}$ or $\Phi^{\neq}$ is a nonempty formula; otherwise, set $\Psi_0 := (-1<0)$ for convenience. Further, set $\varphi_{0,j} := \varphi_j$ for each $j=1,\dots,N$.
    Then clearly for each $z \in \Z^n$,
    \[ \Psi_0(z) \wedge \bigwedge\limits_{j=1}^N \varphi_{0,j}(z) \iff \Phi(z) \]
    holds and implies the desired equisatisfiability.
    
    For the induction step, let $i \in \{1,\dots,N\}$ and suppose that $\Psi_k$ and $\varphi_{k,j}$ have the described properties for each $k = 0,\dots,i-1$ and $j=1,\dots,N$.
    Consider $\varphi_{i-1,i}$, which we may scale to clear all denominators from the rational coefficients to obtain $p=(p_1,\dots,p_{n+1}) \in \Z^{n+1}$ such that for $z \in \Z^n$,
    \[ \varphi_{i-1,i}(z) \iff p\cdot (z,1) = 0 . \]
    
    If $p_j = 0$ for each $j=1,\dots,n$, then $\varphi_{i-1,i}(x)$ is variable-free. Further if $p_{n+1} \neq 0$, then $\varphi_{i-1,i}$ is false. So $\exists x\in \Z^n\ \Phi^=(x) \not\in \sinPA_{\bf D}$ and we set $\Psi := (1 < 0)$ to complete the proof.
    Otherwise if $p_{n+1} = 0$, then $\varphi_{i-1,i}$ is true so we simply set $\Psi_i := \Psi_{i-1}$ and $\varphi_{i,j} := \varphi_{i-1,j}$ for each $j=1,\dots,N$ to proceed. By the induction hypothesis and since $\varphi_{i-1,i}$ is true in this case, the desired equisatisfiability holds.
    
    If instead $p_j \neq 0$ for some $j \in \{1,\dots,n\}$, then let $k$ be such that $|p_k|$ is nonzero yet minimal.
    Define $q = (q_1,\dots,q_{n+1}) \in \Q^{n+1}$ by $q_k := 0$ and $q_j := - p_j / p_k$ for each $j\neq k$ so that for $z \in \Z^n$,
    \[ \varphi_{i-1,i}(z) \iff z_k = q\cdot(z,1) . \]
    If $|p_k| = 1$, then $q \in \Z^{n+1}$. So $q\cdot(z,1) \in \Z$ for $z \in \Z^n$.
    Thus we define $\Psi_{i}$ by imbuing $\Psi_{i-1}$ with the constraint of $\varphi_{i-1,i}$, replacing the variable $x_k$ as follows:
    \[ \Psi_{i}(x) := \Psi_{i-1}\big(x_1,\dots,x_{k-1}, q\cdot(x,1), x_{k+1},\dots,x_n \big) . \]
    
    Otherwise if $|p_k| \geq 2$, we introduce a divisibility predicate to $\Psi_i$ to enforce that the replacement term $q\cdot(x,1)$ takes on integer values.
    Indeed, $q\cdot(z,1)$ is an integer for $z \in \Z^n$ if and only if $|p_k|$ divides $p'\cdot (z,1)$, where $p'=(p'_1,\dots,p'_{n+1}) \in \Z^{n+1}$ is defined by $p'_k := 0$ and $p'_j := -p_j$ for $j\neq k$.
    So we replace $x_k$ as in the first case and now also append a divisibility predicate to ensure the replacement term is integral. We set
    \[ \Psi_i(x) := \Psi_{i-1}\big(x_1,\dots,x_{k-1}, q\cdot(x,1), x_{k+1},\dots,x_n \big) \wedge D_{|p_k|}\big( p'\cdot(x,1) \big) . \]
    Despite having eliminated appearances of $x_k$ in either case, we still write all of $x$ as the free variables of $\Psi_i$. That is, we have
    \[ \varphi_{i-1,i}(x) \rightarrow \big( \Psi_{i-1}(x) \leftrightarrow \Psi_i(x)\big) \in \sinPA_{\bf D} . \]
    
    We also define $\varphi_{i,j}$ for each $j=i+1,\dots,N$ by the same variable replacement. That is, for $j=i+1,\dots,N$ we set
    \[ \varphi_{i,j}(x) := \varphi_{i-1,j}(x_1,\dots,x_{k-1}, q\cdot(x,1), x_{k+1}, \dots, x_n) \]
    so that
    \[ \varphi_{i-1,i}(x) \rightarrow \big( \varphi_{i-1,j}(x) \leftrightarrow \varphi_{i,j}(x)\big) \in \sinPA_{\bf D} \]
    holds for each $j = i+1,\dots,N$.
    By the induction hypothesis,
    \begin{align*}
        &\exists x\in\Z^n\ \Phi(x) \in \sinPA_{\bf D} \text{ if and only if } \\
        &\exists x\in\Z^n\ \Psi_{i-1}(x) \wedge \varphi_{i-1,i}(x) \wedge \bigwedge\limits_{j=i+1}^N \varphi_{i-1,j}(x) \in \sinPA_{\bf D} .
    \end{align*}
    By the above construction of $\Psi_i$ and each $\varphi_{i,j}$, we have that for $z \in \Z^n$,
    \begin{align*}
      \Psi_{i-1}(z) \wedge \varphi_{i-1,i}(z) \wedge \bigwedge\limits_{j=i+1}^N \varphi_{i-1,j}(z) \text{ holds}\\
      \text{if and only if } \Psi_i(z) \wedge \varphi_{i-1,i}(z) \wedge \bigwedge\limits_{j=i+1}^N \varphi_{i,j}(z) &\text{ holds} .
    \end{align*}
    Then since $x_k$ only appears in $\varphi_{i-1,i}$ and the term $q\cdot(x,1)$ used for replacement is either integral by way of $q\in \Z^{n+1}$ or a divisibility predicate of $\Psi_i$, we have
    \begin{align*}
      \exists x\ \Psi_i(x) \wedge \varphi_{i-1,i}(x) \wedge \bigwedge\limits_{j=i+1}^N \varphi_{i,j}(x) &\in \sinPA_{\bf D} \\
      \text{if and only if } \exists x\ \Psi_i(x) \wedge \bigwedge\limits_{j=i+1}^N \varphi_{i,j}(x) &\in \sinPA_{\bf D} .
    \end{align*}
    
    Hence, induction is complete. Now set $\Psi := \Psi_N$ and observe that $\Psi$ is a positive Boolean combination of $\mathcal{L}_{\sin}$-inequalities and divisibility predicates such that
    \[ \exists x \in \Z^n \ \Phi(x) \in \sinPA_{\bf D} \text{ if and only if } \exists x\in\Z^n\ \Psi(x) \in \sinPA_{\bf D} . \qedhere \]
\end{proof}

\begin{thm}\label{thm:elim-lin-from-inequality}
Let $\Phi(x)$ be a conjunction of $\mathcal{L}_{\sin}$-inequalities and divisibility predicates.
Then there is a positive Boolean combination $\Psi(x)$ of oscillatory $\mathcal{L}_{\sin}$-inequalities and divisibility predicates such that
\[ \exists x\in\Z^n\ \Phi(x) \in \sinPA_{\bf D} \text{ if and only if } \exists x\in\Z^n\ \Psi(x) \in \sinPA_{\bf D} . \]
Moreover, $\Psi$ can be computed from $\Phi$.
\end{thm}
\begin{proof}
    We proceed by induction on the count $d$ of variables from $x$ appearing with nonzero coefficient on the left-hand side of an $\mathcal{L}_{\sin}$-inequality from $\Phi(x)$; we refer to such variables as appearing linearly.
    The main observation used in this proof is that oscillatory $\mathcal{L}_{\sin}$-terms are bounded,
    so only finitely many level sets, each of which will be characterized by an $\mathcal{L}$-equality, are needed to obtain a matrix formula which is $\sinPA_{\bf D}$-equisatisfiable to $\Phi$. These level sets enable variable replacement to eradicate the linear appearances of variables.
    The presence of divisibility predicates complicates, but does not obstruct, this approach.
    
    Setting $\Psi := \Phi$ suffices for the case with $d=0$ since every $\mathcal{L}_{\sin}$-inequality is already oscillatory.
    
    \noindent For the induction step, let $d > 0$ and write each inequality from $\Phi(x)$ in the form \[ q_j \cdot (x,1) < t_j(x) \] for $j=1,\dots,M$ where $q_j \in \Q^{n+1}$, $t_j$ is an oscillatory $\mathcal{L}_{\sin}$-term, and $M > 0$.
    Without loss of generality, suppose $q_{j,n} \neq 0$ for some $j$; that is, assume $x_n$ appears linearly in $\Phi(x)$.
    Set
    \[ N := \lcm\Big(\Big\{ \frac{k}{\gcd(k,|p_n|)} \in \Z_{>0} :
    D_k\big(p\cdot(x,1)\big) \text{ is a conjunct of } \Phi \text{ with } p_n \neq 0 \Big\}\Big) . \]
    For $y \in \Z$, denote \[ z^y := (z_1,\dots,z_{n-1}, z_n + Ny ) . \]
    Suppose a divisibility predicate $D_k\big(p\cdot (x,1)\big)$ from $\Phi$ holds for some $z \in \Z^n$. By construction of $N$, we have that $k$ divides $p_nN$. Hence $k$ divides $p\cdot(z^y,1) = p\cdot(z,1) + p_nNy$. Thus if $z$ satisfies every divisibility predicate of $\Phi$, each $z^y$ does as well.
    
    Set $J := \{ j \in \{1,\dots,M\} : q_{j,n} \neq 0 \}$. For each $j \in J$, let $S_j$ be the finite set according to Fact \ref{fact:finiteset} with the function $x \mapsto q_j \cdot (x,1)$ as $f$ and \[ I = \big[ -R(t_j) - N|q_{j,n}|,\ R(t_j) \big) , \]
    where $R(t_j)$ denotes the radius of $t_j$ from Definition \ref{def:radius}.
    We now show that $\exists x\in\Z^n\ \Phi(x) \in \sinPA_{\bf D}$ if and only if $\exists x\in\Z^n\ \Phi'(x) \in \sinPA_{\bf D}$, where \[ \Phi'(x) := \Phi(x) \wedge \bigvee\limits_{j \in J} \bigvee_{c \in S_j} \big( q_j\cdot(x,1) = c \big) . \]
    
    \noindent
    The backward implication is clear.
    For the forward direction, let $z \in \Z^n$ and suppose $\Phi(z)$ holds.
    Assume that $\Phi'(z')$ does not hold for any $z' \in \Z^n$; so for every $z' \in \Z^n$, if $\Phi(z')$ holds, then $q_j \cdot(z',1) \neq c$ for every $j \in J$ and $c \in S_j$.
    By construction of $S_j$ and Fact \ref{fact:radius}, we have that $\Phi(z')$ implies
    \[ q_j\cdot(z',1) < -R(t_j) - N|q_{j,n}| \]
    for each $j\in J$.
    Notice that when $q_{j,n}$ is positive (negative), then $q_j\cdot(z^y,1)$ increases as $y$ increases (decreases).
    
    We now chose $y \in \Z$ and $\alpha \in \{-1,1\}$ such that $\Phi(z^y)$ holds while $\Phi(z^{y+\alpha})$ does not.
    Since $\Phi(z)$ holds, there are $y \in \Z$ and $j\in J$ such that $\Phi(z^y)$ holds but $q_j \cdot (z^{y+\alpha},1) \geq t_j(z^{y+\alpha})$, with $\alpha = q_{j,n} / |q_{j,n}| \in \{-1,1\}$.
    The sign of $\alpha$ captures whether the $\mathcal{L}$-term $q_j \cdot (z^y, 1)$ is increasing or decreasing in $y$.
    Since $t_j(z^{y+\alpha}) \geq -R(t_j)$ by Fact \ref{fact:radius}, we have
    \begin{align*}
        q_j\cdot(z^y, 1)
        &= q_j\cdot(z^{y+\alpha},1) - N|q_{j,n}| \\
        &\geq t_j(z^{y+\alpha}) - N|q_{j,n}| \\
        &\geq -R(t_j) - N|q_{j,n}| .
    \end{align*}
    Since $\Phi(z^y)$ holds, this contradicts $q_j \cdot (z^y, 1) < -R(t_j) - N|q_{j,n}|$.
    
    \noindent Let $\Phi_{j,c}(x) := \Phi(x) \wedge \big( q_j\cdot(x,1) = c \big)$ and notice that the conjunct $q_j \cdot (x,1) = c$ characterizes a level set.
    We now have
    \[ \exists x\in\Z^n\ \Phi(x) \in \sinPA_{\bf D} \text{ if and only if } \exists x\in\Z^n\ \bigvee_{j \in J} \bigvee_{c \in S_j} \Phi_{j,c}(x) \in \sinPA_{\bf D} . \]
    Since each $\Phi_{j,c}$ is a conjunction of $\mathcal{L}_{\sin}$-inequalities, $\mathcal{L}$-equalities, and divisibility predicates, we may apply Theorem \ref{thm:elim-equality} to obtain a positive Boolean combination $\Psi_{j,c}(x)$ of $\mathcal{L}_{\sin}$-inequalities and divisibility predicates such that for $z\in \Z^n$, $\Phi_{j,c}(z) \in \sinPA_{\bf D}$ if and only if $\Psi_{j,c}(z) \in \sinPA_{\bf D}$.
    By inspection of the proof of Theorem \ref{thm:elim-equality}, every variable appearing linearly in $\Psi_{j,c}$ already did so in $\Phi_{j,c}$. Indeed, $\Psi_{j,c}$ is constructed by replacement of variables, which possibly duplicate preexisting linear appearances of variables, and the introduction of divisibility predicates, which carry no linear appearances of variables.
    In fact since $j \in J$ and $q_{j,n} \neq 0$, the variable replacement yielding $\Psi_{j,c}$ eliminated all appearances of some variable, possibly $x_n$, from $x$.
    Thus $\Psi_{j,c}$ bears strictly fewer than $d$ variables from $x$ appearing linearly.
    
    We may now apply the induction hypothesis to each $\Psi_{j,c}$ to obtain a positive Boolean combination $\Psi'_{j,c}$ of oscillatory $\mathcal{L}_{\sin}$-inequalities and divisibility predicates which is $\sinPA_{\bf D}$-equivalent to $\Psi_{j,c}$.
    From these, we construct
    \[ \Psi(x) :=  \bigvee_{j \in J} \bigvee_{c \in S_j} \Psi'_{j,c}(x) \]
    which is a positive Boolean combination of oscillatory $\mathcal{L}_{\sin}$-inequalities and divisibility predicates such that $\exists x\in\Z^n\ \Phi(x) \in \sinPA_{\bf D}$ if and only if $\exists x\in\Z^n\ \Psi(x) \in \sinPA_{\bf D}$. So the induction is complete.
\end{proof}

\begin{figure}
\includegraphics[scale=0.7]{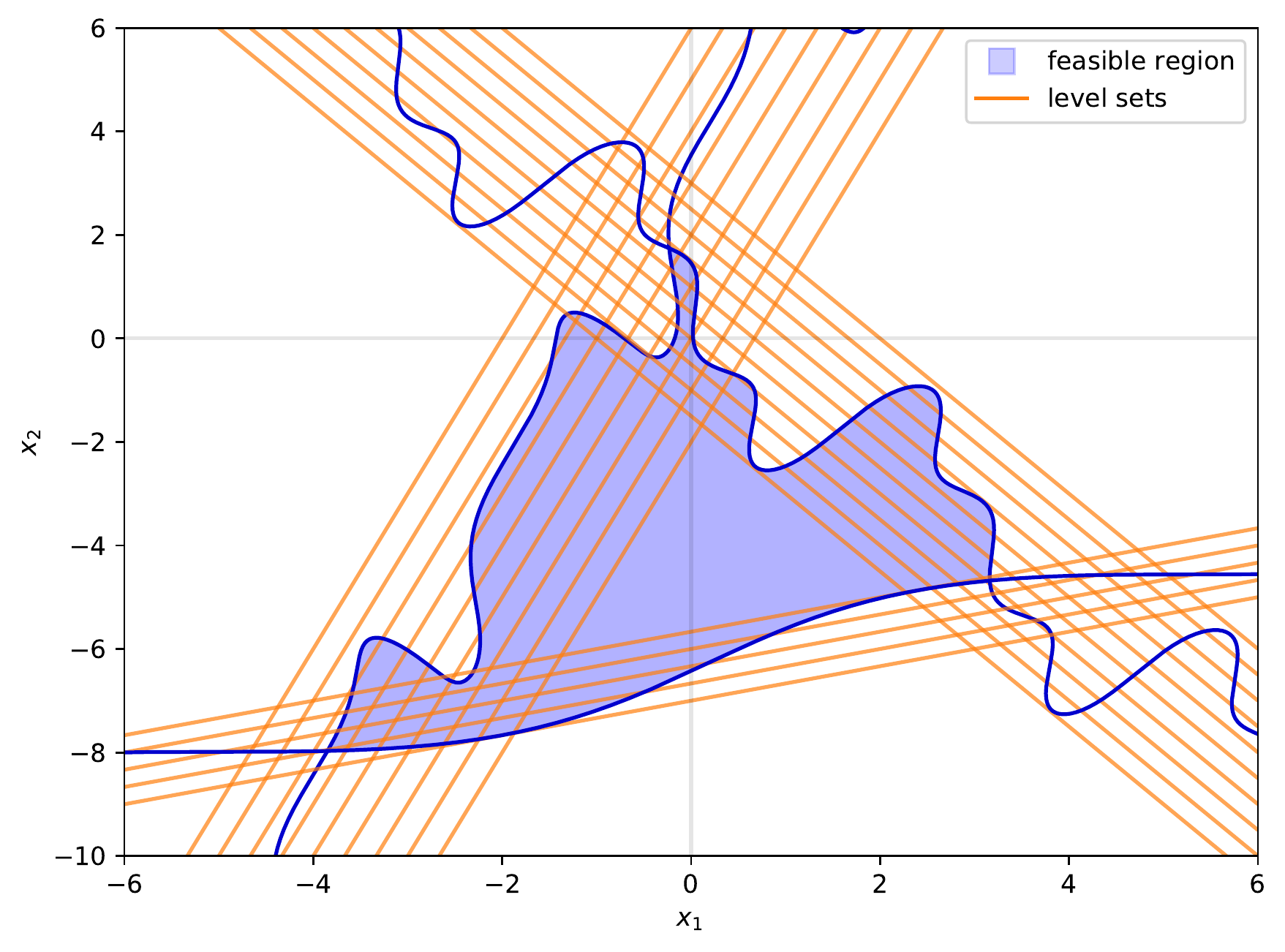}
\caption{Visualization of the level sets computed for multiple $\mathcal{L}_{\sin}$-inequalities in the proof of Theorem \ref{thm:elim-lin-from-inequality}.}\label{fig:Lsin-levelsets}
\end{figure}
\begin{example}
Figure \ref{fig:Lsin-levelsets} portrays the set of real solutions (i.e., the feasible region) to the conjunction of the following $\mathcal{L}_{\sin}$-inequalities over variables $x_1,x_2$:
\begin{align*}
    -3x_1 + x_2 - 2 &< 2\sin\!\big(3x_1 + \sin(x_2 - 1)\big) + \sin\!\left(\frac{1}{2}\right), \\
    2x_1 + \frac{4}{3}x_2 - 1 &< \sin (-3x_1 + 2x_2 - 1) + 2\sin(-2x_1), \\
    \frac{1}{2}x_1 - \frac{3}{2}x_2 - \frac{19}{2} &< -\sin\!\left(\frac{1}{2}x_1 + \frac{1}{3}x_2 + 2\right).
\end{align*}
The orange lines are the level sets for each $\mathcal{L}$-term as computed according to the proof of Theorem \ref{thm:elim-lin-from-inequality}.
Each blue curve is the solution to one of the $\mathcal{L}_{\sin}$-inequalities above reframed as an equality (i.e., having replaced $<$ by $=$).
The shaded blue region is the real solution set to all three $\mathcal{L}_{\sin}$-inequalities.
\end{example}

\begin{thm}\label{thm:elim-divisibility}
Let $\Phi(x)$ be a conjunction of oscillatory $\mathcal{L}_{\sin}$-inequalities and divisibility predicates.
Then there is a positive Boolean combination $\Psi(x)$ of oscillatory $\mathcal{L}_{\sin}$-inequalities such that
\[ \exists x\in\Z^n\ \Phi(x) \in \sinPA_{\bf D} \text{ if and only if } \exists x \in \Z^n\ \Psi(x) \in \sinPA . \]
Moreover, $\Psi$ can be computed from $\Phi$.
\end{thm}
\begin{proof}
    We partition the atoms of $\Phi$ into conjunctions $\Phi^D$ and $\Phi^<$ of divisibility predicates and oscillatory $\mathcal{L}_{\sin}$-inequalities, respectively. That is,
    \[ \Phi =: \Phi^D \wedge \Phi^{<} . \]
    We proceed by induction on the count $d$ of variables from $x$ appearing with nonzero coefficient in $\Phi^D(x)$. The base case with $d=0$ is trivial since $\Phi^D$ is variable-free and may be decided on its own. If $\Phi^D$ holds, then $\Psi := \Phi$ suffices; otherwise set $\Psi := (1 < 0)$.
    
    \noindent For the induction step, let $d > 0$.
    Without loss of generality, suppose $x_n$ appears with nonzero coefficient in some divisibility predicate of $\Phi^D$.
    Set
    \[ N := \lcm\big( \big\{ k \in \Z_{\geq 2} :
    D_k\big(p\cdot(x,1)\big) \text{ is a conjunct of } \Phi^D \text{ with } p_n \neq 0 \big\} \big) . \]
    Let $D_k\big( p\cdot(x,1)\big)$ be a conjunct of $\Phi^D$. Then $k$ divides $N$. Hence $x_n \equiv j \mod N$ implies $x_n \equiv j \mod k$. Thus for each $j=0,\dots,N-1$, we have that
    \[ D_N(x_n - j) \rightarrow \big( \Phi^D(x) \leftrightarrow \Phi^D(x_1,\dots,x_{n-1},j)\big) \in \sinPA_{\bf D} . \]
    Notice $\bigvee_{j=0}^{N-1} D_N(x_n - j) \in \sinPA_{\bf D}$. Then
    \begin{align*}
    \Phi(x) &\leftrightarrow \Big( \bigvee_{j=0}^{N-1} \big( D_N(x_n - j) \wedge \Phi(x) \big) \Big) \in \sinPA_{\bf D} \text{, and thus} \\
    \Phi(x) &\leftrightarrow \Big( \bigvee_{j=0}^{N-1} \big( D_N(x_n - j) \wedge \Phi^D(x_1,\dots,x_{n-1}, j) \wedge \Phi^<(x) \big) \Big) \in \sinPA_{\bf D} . \end{align*}
    
    Finally, observe for $j \in \{0,\dots,N-1\}$ that $D_N(x_n - j)$ holds if and only if $x_n = Ny + j$ for some $y \in \Z$.
    With this, we replace $x_n$ by $Nx_n + j$ in $\Phi^<(x)$ to linearly project the solution set of each disjunct above along the $n^\text{th}$ coordinate. Since $D_N(Nx_n + j - j) \in \sinPA_{\bf D}$ for every $j$, we may then omit the divisibility predicate determining the residue of $x_n$.
    The result is $\sinPA_{\bf D}$-equisatisfiable to the formula before replacement. That is, for each $j=0,\dots,N$ we have
    \begin{align*}
    &\exists x\in\Z^n\ D_N(x_n - j) \wedge \Phi^D(x_1,\dots,x_{n-1}, j) \wedge \Phi^<(x) \in \sinPA_{\bf D} \text{ if and only if}\\
    &\exists x\in\Z^n\ \Phi^D(x_1,\dots,x_{n-1}, j) \wedge \Phi^<(x_1,\dots,x_{n-1}, Nx_n + j) \in \sinPA_{\bf D} .
    \end{align*}
    Since $x_n$ has been eliminated from the divisibility predicates in the matrix formula of the latter sentence above, we may invoke the induction hypothesis for each $j=0,\dots,N-1$ to obtain a positive Boolean combination $\Psi_j(x)$ of oscillatory $\mathcal{L}_{\sin}$-inequalities that is $\sinPA_{\bf D}$-equisatisfiable to \[ \Phi^D(x_1,\dots,x_{n-1}, j) \wedge \Phi^<(x_1,\dots,x_{n-1}, Nx_n + j) . \]
    From these, we construct
    \[ \Psi(x) := \bigvee_{j=0}^{N-1} \Psi_j(x) \]
    which is a positive Boolean combination of oscillatory $\mathcal{L}_{\sin}$-inequalities (and more generally, a $\sin$-PA formula) such that $\exists x\in\Z^n\ \Phi(x) \in \sinPA_{\bf D}$ if and only if $\exists x \in \Z^n\ \Psi(x)$ is true. So the induction is complete.
\end{proof}

\subsection{Reduction to the real additive group with sine.}
The goal of this subsection is first to prove Theorem \ref{thm:decide-osc-inequality}, which passes the decision procedure for existential $\sin$-PA sentences with a restricted form to one for the theory of the ordered additive group of real numbers extended by sine. Then, we finally prove Theorem B.

\begin{thm}\label{thm:decide-osc-inequality}
Let $\Phi(x)$ be a positive Boolean combination of oscillatory $\mathcal{L}_{\sin}$-inequalities.
Then there is an existential $\mathcal{L}_{\sin}$-sentence $\theta$ such that
\[ \exists x \in \Z^n\ \Phi(x) \text{ is true if and only if } \left(\R, <, +, \sin\right) \models \theta . \]
Moreover, $\theta$ can be computed from $\Phi$.
\end{thm}

\begin{defn}
  Let $N \in \Z_{>0}$. Then $\equiv_N$ is the equivalence relation defined by
  \begin{align*}
    (\alpha_1,\dots,\alpha_n) \equiv_N (\beta_1,\dots,\beta_n) &\text{ if and only if for each } i=1,\dots,n \\ \text{there exists } k_i \in \Z &\text{ such that } \alpha_i = \beta_i + 2N\pi k_i .
  \end{align*}
 Note that the equivalence classes of $\equiv_N$ are precisely the cosets of $(2N\pi\Z)^n$ in $\R^n$.
\end{defn}

\begin{lem}\label{lem:osc_equiv}
  Let $\varphi(x)$ be an oscillatory $\mathcal{L}_{\sin}$-inequality. The set of real solutions to $\varphi(x)$ is a union of $\equiv_N$ classes, where
  \begin{align*}
    N := \lcm\Big( \Big\{ b \in \Z_{>0} : \, &\frac{a}{b} \in \Q \text{ is a coefficient for some $x_i$ in an $\mathcal{L}_{\sin}$-subterm} \\
    &\text{ of $t(x)$ such that $a \in \Z$, $b \in \Z_{>0}$, and $\gcd(a,b) = 1$} \Big\} \Big) .
  \end{align*}
\end{lem}
\begin{proof}
    The main observation for this proof is that every appearance of a variable in an oscillatory $\mathcal{L}_{\sin}$-inequality is under the application of sine. To this end, recall that $\sin(\alpha + 2\pi) = \sin \alpha$ for every $\alpha \in \R$.
    
    Now let $t(x)$ be the oscillatory $\mathcal{L}_{\sin}$-term and $q \in \Q$ such that $\varphi(x)$ is $q < t(x)$. For $z=(z_1,\dots,z_n)\in \R^n$
    and $j \in \{1,\dots,n\}$, we denote
    \[ z^N_j := (z_1,\dots,z_{j-1}, z_j + 2N\pi, z_{j+1},\dots,z_n) . \]
    We claim that $t(z^N_j) = t(z)$ for every $z\in \R^n$ and $j=1,\dots,n$.
    To prove this, we proceed by induction on the sine depth $d$ of $t(x)$.
    The case with $d=0$ is trivial as $t$ does not depend on $x$.
    
    \noindent For the induction step, let $d > 0$ and fix $j \in \{1,\dots,n\}$.
    Writing $t(x)$ in the form of \eqref{eq:t_form}, fix $p_0 \in \Q^{n+1}$, $p_1,\dots,p_K \in \Q^{m+n+1}$, $r_1,\dots,r_K \in \Q_{\neq 0}$, $K \geq 0$, $m\geq 0$, and $\mathcal{L}_{\sin}$-terms $t_i(x)$ with $q_i \in \Q^{n+i}$ for $i=1,\dots,m$.
    Since $t(x)$ is oscillatory, each entry of $p_0$ is 0.
    Critically, if $\alpha \in \Q$ is the coefficient on a variable in $t(x)$ or any $t_k(x)$, then $\alpha N \in \Z$ by construction of $N$.
    Since the sine depth of each $\mathcal{L}_{\sin}$-term $t_k(x)$ is strictly less than $d$, we may apply the induction hypothesis to $t_k(x)$. We obtain that $t_k(z^N_j) = t_k(z)$ for every $z\in \R^n$ and $k=1,\dots,m$, because the construction of $N$ considers the coefficients appearing in $t_1,\dots,t_m$ as $\mathcal{L}_{\sin}$-subterms of $t(x)$.
    Thus for every $z\in \R^n$
    \begin{align*}
        t(z^N_j)
        &= \sum\limits_{i=1}^K r_i \sin \Big( p_i \cdot \big(z^N_j, 1, t_1(z^N_j), \dots, t_m(z^N_j)\big) \Big) \\
        &= \sum\limits_{i=1}^K r_i \sin \Big( 2p_{i,j}N\pi + p_i \cdot \big(z, 1, t_1(z), \dots, t_m(z)\big) \Big) \\
        &= \sum\limits_{i=1}^K r_i \sin \Big( p_i \cdot \big(z, 1, t_1(z), \dots, t_m(z)\big) \Big)
        = t(z)
    \end{align*}
    since $p_{i,j}N \in \Z$. So induction is complete.
    Then since $t(z^N_j) = t(z)$ for all $z\in \R^n$ and $j \in \{1,\dots,n\}$, the set of real solutions to $\varphi(x)$ is a union of $\equiv_N$ classes.
\end{proof}

\begin{defn}\label{def:D}
Let $N \in \Z_{>0}$. Then $D_N$ is the set of representatives in $[0,2N\pi)$ for integers modulo $2N\pi$. That is,
\[ D_N := \big\{ X - 2N\pi Z : X,Z\in \Z \text{ and } 2N\pi Z \leq X < 2N\pi (Z+1) \big\} . \]
\end{defn}

\begin{lem}\label{lem:proxy-formula}
    Let $\Phi(x)$ be a positive Boolean combination of oscillatory $\mathcal{L}_{\sin}$-inequalities. Then there is some $N \in \Z_{>0}$ such that
    \[ \exists x \in \Z^n\ \Phi(x) \text{ is true if and only if } \left(\R, <, +, \sin, D_N \right) \models\ \exists x\ \Phi(x) \wedge \bigwedge_{i=1}^n D_N(x_i) . \]
\end{lem}
\begin{proof}
    Let $\varphi_1,\dots,\varphi_M$ enumerate the inequalities of $\Phi$.
    By Lemma \ref{lem:osc_equiv}, let $N_i \in \Z_{>0}$ be such that the real solution set to $\varphi_i(x)$ is a union of $\equiv_{N_i}$ classes for each $i=1,\dots,M$.
    Set $N := \lcm( N_1,\dots,N_M )$.
    So the real solution set to each $\varphi_i(x)$ is a union of $\equiv_N$ classes; taking unions and intersections of such solution sets preserves this property.
    That is, the real solution set to $\Phi(x)$ is a union of $\equiv_N$ equivalence classes.
    
    Suppose $\exists x \in \Z^n\ \Phi(x)$ is true. So let $z\in \Z^n$ be such that $\Phi(z)$ holds. By Definition \ref{def:D}, there is $y\in D_N^n$ such that $z \equiv_N y$. Then since $\Phi(z)$ holds, so does $\Phi(y)$.
    
    Now suppose instead that $\left(\R, <, +, \sin, D_N \right)$ models $\exists x\ \Phi(x) \wedge \bigwedge_{i=1}^n D_N(x_i)$. So let $y \in D_N^n$ be such that $\Phi(y)$ holds. By Definition \ref{def:D}, there is $z\in \Z^n$ such that $z \equiv_N y$. Then since $\Phi(y)$ holds, so does $\Phi(z)$. So $\exists x \in \Z^n\ \Phi(x)$ is true.
\end{proof}
\begin{defn}\label{def:proxy}
Let $\Phi(x)$ be a positive Boolean combination of oscillatory $\mathcal{L}_{\sin}$-inequalities. Then the {\bf proxy solution set} of $\Phi$ is \[ M_{\Phi} := \{ z\in [0,2N\pi)^n : \left(\R, <, +, \sin\right) \models \Phi(z) \} , \]
where $N$ is defined from $\Phi$ as in the statement of Lemma \ref{lem:osc_equiv}. 
\end{defn}

\begin{lem}\label{lem:proxy_open}
  Let $\Phi(x)$ be a positive Boolean combination of oscillatory $\mathcal{L}_{\sin}$-inequalities and $N$ be defined from $\Phi$ as in the statement of Lemma \ref{lem:osc_equiv}.
  Then $M_\Phi$ is open in $[0,2N\pi)^n$.
\end{lem}
\begin{proof}
    We claim that for any oscillatory $\mathcal{L}_{\sin}$-inequality $\varphi(x)$ from $\Phi$, the set
    \[ M_{\varphi,N} := \{ z \in [0,2N\pi)^n : \left(\R, <, +, \sin\right) \models \varphi(z) \} \]
    is open in $[0,2N\pi)^n$.
    Since $\varphi$ is an oscillatory $\mathcal{L}_{\sin}$-inequality, let $q_{n+1} \in \Q$ and $t(x)$ be an oscillatory $\mathcal{L}_{\sin}$-term such that $\varphi(x)$ is $q_{n+1} < t(x)$.
    In particular, $\varphi$ is a strict inequality. Since the sine function is continuous, we have that $M_{\varphi,N}$ is open in $[0,2N\pi)^n$.
    
    Now, observe that $M_\Phi$ is a finite, positive Boolean combination (i.e., comprising finitely many unions and intersections) of sets with the form $M_{\varphi,N}$ for oscillatory $\mathcal{L}_{\sin}$-inequalities $\varphi$ from $\Phi$.
    Since each such $M_{\varphi,N}$ is open in $[0,2N\pi)^n$, we have that $M_\Phi$ is as well.
\end{proof}

\begin{proof}[Proof of Theorem \ref{thm:decide-osc-inequality}]
    We claim that $\exists x \in \Z^n\ \Phi(x)$ is true if and only if $M_{\Phi}$ is nonempty.
    Let $N$ be defined from $\Phi$ as in the statement of Lemma \ref{lem:osc_equiv}.
    Since $D_N$ is dense in $[0,2N\pi)$, the set $D_N^n$ is dense in $[0,2N\pi)^n$.
    Then by Lemma \ref{lem:proxy-formula}, Lemma \ref{lem:proxy_open}, and Definition \ref{def:proxy} we have that
    \begin{align*}
        \exists x \in \Z^n\ \Phi(x) \text{ is true }
        \text{if and only if } &\Phi(x) \text{ has a solution in } D_N^n \\
        \text{if and only if } &\Phi(x) \text{ has a solution in } [0,2N\pi)^n \\
        \text{if and only if } &M_{\Phi} \text{ is nonempty} .
    \end{align*}
    
    Define the unary $\mathcal{L}_{\sin}$-formulas
    \begin{align*}
        \chi_{[0,2\pi)}(x) &:= (0 \leq x) \land (x < 7) \land (x > 6 \rightarrow \sin x < 0) \text{ and} \\
        \chi_{[0,2N\pi)}(x) &:= \exists y\ \chi_{[0,2\pi)}(y) \wedge x = Ny .
    \end{align*}
    It is easy to see that $\chi_{[0,2\pi)}(Z)$ holds if and only if $Z \in [0,2\pi)$; similarly $\chi_{[0,2N\pi)}(Y)$ holds if and only if $Y \in [0,2N\pi)$.
    Now consider the $\mathcal{L}_{\sin}$-formula
    \[ \chi_{M_{\Phi}}(x) := \Phi(x) \wedge \bigwedge\limits_{i=1}^n \chi_{[0,2N\pi)}(x_i) . \]
    For $z \in \R^n$, the formula $\chi_{M_{\Phi}}(z)$ clearly holds if and only if $z \in [0,2N\pi)^n$ and $\left(\R, <, +, \sin\right)\models \Phi(z)$.
Set
    \[ \theta := \exists x\ \chi_{M_{\Phi}}(x) \]
    Then $\left(\R, <, +, \sin\right)\models \theta$ 
    if and only if $M_{\Phi}$ is nonempty, as desired.
\end{proof}

\begin{proof}[Proof of Theorem B]
    Let $\Phi(x)$ be a quantifier-free $\mathcal{L}_{\sin}$-formula. To decide the existential $\sin$-PA sentence
    \[ \exists x \in \Z^n\ \Phi(x) , \]
    we may assume without loss of generality that $\Phi$ is a positive Boolean combination of $\mathcal{L}_{\sin}$-inequalities, -equalities, and -disequalities.
    We will adjust formulas into disjunctive normal form multiple times; notice that this process does not introduce negations and moreover preserves the types of $\mathcal{L}_{\sin}$-literal that are present.
    
    Apply Theorem \ref{thm:elim-Lsin-equality} to each $\mathcal{L}_{\sin}$-equality of $\Phi(x)$ to obtain equivalent positive Boolean combinations $\Psi_1(x),\dots,\Psi_{M_0}(x)$ of $\mathcal{L}$-equalities and -disequalities.
    Let $\Phi^{<,\neq}$ be the conjunction of $\mathcal{L}_{\sin}$-inequalities and -disequalities of $\Phi$ and set
    \[ \Phi^1(x) := \Phi^{<,\neq}(x) \wedge \bigwedge\limits_{i=1}^{M_0} \Psi_i(x). \]
    Thus $\exists x \in \Z^n\ \Phi(x)$ is true if and only if $\exists x \in \Z^n\ \Phi^1(x)$ is true.
    
    Adjust $\Phi^1(x)$ into disjunctive normal form.
    Apply Theorem \ref{thm:elim-equality} to each conjunctive clause of $\Phi^1(x)$ to obtain equivalent positive Boolean combinations $\Psi^1_1(x), \dots,\linebreak \Psi^1_{M_1}(x)$ of $\mathcal{L}_{\sin}$-inequalities and divisibility predicates, which are thus $\mathcal{L}_{\sin,\Z, D}$-formulas. Set
    \[ \Phi^2(x) := \bigwedge\limits_{i=1}^{M_1} \Psi^1_i(x). \]
    Thus $\exists x \in \Z^n\ \Phi^1(x) \in \sinPA$ if and only if $\exists x\in\Z^n\ \Phi^2(x) \in \sinPA_{\bf D}$. Here, we recall that since $\sinPA$ is a subtheory of $\sinPA_{\bf D}$ and $\Phi^1$ is an $\mathcal{L}_{\sin}$-formula, $\exists x \in \Z^n\ \Phi(x) \in \sinPA$ if and only if $\exists x\in \Z^n\ \Phi(x) \in \sinPA_{\bf D}$.
    
    Adjust $\Phi^2(x)$ into disjunctive normal form.
    Apply Theorem \ref{thm:elim-lin-from-inequality} to each conjunctive clause of $\Phi^2(x)$ to obtain $\sinPA_{\bf D}$-equisatisfiable positive Boolean combinations $\Psi^2_1(x),\dots,\Psi^2_{M_2}(x)$ of oscillatory $\mathcal{L}$-inequalities and divisibility predicates. Set
    \[ \Phi^3(x) := \bigvee\limits_{i=1}^{M_2} \Psi^2_i(x) \]
    so that $\exists x\in\Z^n\ \Phi^2(x) \in \sinPA_{\bf D}$ if and only if $\exists x\in\Z^n\ \Phi^3(x) \in \sinPA_{\bf D}$.
    
    Adjust $\Phi^3(x)$ into disjunctive normal form.
    Apply Theorem \ref{thm:elim-divisibility} to each conjunctive clause of $\Phi^3(x)$ to obtain $\sinPA_{\bf D}$-equisatisfiable positive Boolean combinations $\Psi^3_1(x),\dots,\Psi^3_{M_3}(x)$ of oscillatory $\mathcal{L}_{\sin}$-inequalities, which are thus $\sin$-PA formulas. Set
    \[ \Phi^4(x) := \bigvee\limits_{i=1}^{M_3} \Psi^3_i(x). \]
    Thus $\exists x\in\Z^n\ \Phi^3(x) \in \sinPA_{\bf D}$ if and only if $\exists x \in \Z^n\ \Phi^4(x) \in \sinPA$.
    
    Apply Theorem \ref{thm:decide-osc-inequality} to $\Phi^4(x)$ to obtain an existential $\mathcal{L}_{\sin}$-sentence $\theta$ such that
    \[ \exists x \in \Z^n\ \Phi^4(x) \text{ is true if and only if } \left(\R, <, +, \sin\right) \models \theta . \]
    Finally, recognize that the original $\sin$-PA sentence $\exists x \in \Z^n\ \Phi(x)$ is true if and only if $\left(\R, <, +, \sin\right) \models \theta$. By Theorem \ref{thm:decide-additive-gp}, the latter is decidable under Schanuel's conjecture, which completes the decision procedure.
\end{proof}

\section{Conclusion}\label{sec:conclusion}

In this paper, we have considered and solved decision problems for certain sets of $\sin$-PA sentences. We showed that under a conjecture which is far out of current technology's reach, existential $\sinPA$ can be decided. While we do not see a way forward to remove the use of Schanuel's conjecture, a systematic study of decision problems for subsets of all existential $\sin$-PA sentences is surely desirable. The work of Anai and Weispfenning in \cite{MR1805108} already provides an example of a nontrivial subset that can be decided without any number-theoretic conjectures.\newline

\noindent Of course, it is natural to replace sine by other functions and consider similar questions. For a function $f:\R \to \R$, we define $f$-PA sentences analogously to $\sin$-PA sentences, replacing sine by $f$. This is particularly interesting for well-behaved analytic functions, like logarithms and exponential functions. Even then, we see behavior distinct from $\sinPA$. We contrast Theorem B against the following fact.
\begin{prop}\label{fact:log-PA}
The set of existential $\log$-PA sentences is undecidable.
\end{prop}
\begin{proof}
    We will show how the existence of an integer solution to a given Diophantine equation may be effectively encoded as an existential $\log$-PA sentence. By the negative answer to Hilbert's 10th problem, any set containing the resulting sentences must be undecidable.\newline
    \noindent A Diophantine equation is an $m$-variate polynomial equation with coefficients in $\N$ and $m > 0$, without loss of generality.
    For each monomial $k\prod_{i=1}^m x_i^{n_i}$ with $k \in \N$ and each $n_i \in \N$, we assign a fresh variable $\alpha := \prod_{i=1}^m x_i^{n_i}$ and replace the monomial by $k\alpha$ in the equation. We append the constraints $\alpha, x_1,\dots,x_m > 0$ as conjuncts, then also the following equality which holds if and only if $\alpha = \prod_{i=1}^m x_i^{n_i}$ and $\alpha, x_1,\dots,x_m > 0$:
    \[ \log \alpha = \sum\limits_{i=1}^m n_i\log x_i . \]
    
    \noindent After replacing each monomial as above, we existentially quantify variables over $\Z$; the result is an existential $\log$-PA sentence which is true if and only if the Diophantine equation has a solution over the positive natural numbers.
    Allowing for all integer solutions is a tedious but easy extension.
\end{proof}

\noindent However, when we replace sine by the function $x\mapsto 2^x$ (short: $2^x$), we have the following theorem of Semenov \cite{MR703597} (see also Point \cite{Point}).
\begin{fact}\label{fact:2exp-PA}
    The theory $\FO(\N, <, +, 2^x)$ is decidable.
\end{fact}
\noindent Of course, this gives that the set of existential $2^x$-PA sentences (and even the set of all $2^x$-PA sentences) is decidable, which may be surprising in the light of Fact \ref{fact:log-PA}. We do not know whether this still holds when we consider the usual exponential function $\exp$ instead of $2^x$. It is well-known that the theory $\FO\left(\R,<,+,\exp, \Z \right)$ is undecidable, since multiplication on $\R$ can be defined using the equality $\exp(x)\cdot\exp(y)=\exp(x+y)$. However, it is an open question which sets of $\exp$-PA sentences are decidable.


\bibliographystyle{abbrv}
\bibliography{sinPA}

\end{document}